\documentclass{article}

\usepackage{PRIMEarxiv}

\usepackage[utf8]{inputenc} 
\usepackage[T1]{fontenc}    
\usepackage{hyperref}       
\usepackage{url}            
\usepackage{booktabs}       
\usepackage{amsfonts}       
\usepackage{amsmath,amssymb}
\usepackage{mathtools}
\usepackage{nicefrac}       
\usepackage{microtype}      
\usepackage{lipsum}
\usepackage{fancyhdr}       
\usepackage{graphicx}       
\usepackage{amsthm}
\usepackage{times}
\usepackage{subfig}
\usepackage{xcolor}
\usepackage{array}
\usepackage{cite}
\usepackage{multirow}
\usepackage{adjustbox}
\usepackage[scr=dutchcal]{mathalfa}
\usepackage[font=small,labelfont=bf]{caption}
\usepackage{enumitem}
\usepackage{tikz}
\usetikzlibrary{arrows, arrows.meta, bending, matrix}

\DeclareMathOperator{\diag}{diag}

\newtheorem{theorem}{Theorem}
\newtheorem{lemma}{Lemma}
\newtheorem{proposition}{Proposition}
\newtheorem{corollary}{Corollary}

\newcommand{\e}{\mbox{e}}

\title{
Boundedness of solutions in feedback systems \\ with antithetic controllers
} 

\author{
  Moh Kamalul Wafi$^1$, Arthur C. B. de Oliveira$^1$, and Eduardo D. Sontag$^{1,2}$ \\
  $^1$Department of Electrical and Computer Engineering. \\
  $^2$Department of BioEngineering and affiliated with the Departments of Chemical Engineering and Mathematics. \\
  Northeastern University \\
  Boston, USA\\
  \texttt{\{wafi.m, a.castello, e.sontag\}@northeastern.edu}
}

\begin{document}
\maketitle

\begin{abstract}
Antithetic feedback controllers have become a key experimental and theoretical tool in synthetic biology.
Introduced by Khammash and collaborators about 10 years ago, they are employed in order to achieve the practical regulation of protein expression, including tracking and robust disturbance rejection.
In closed-loop, there are unique equilibria which, depending on parameter values, can be unstable.
It had been shown, however, that this instability is not arbitrary: any bounded trajectory that stays away from the equilibrium must converge to a periodic orbit.
This motivated a long-standing open question: is every trajectory bounded?
In other words, even if the equilibrium is unstable, can nonlinear effects prevent unbounded excursions in the state space?
This paper provides an affirmative answer, establishing the boundedness of all solutions.
Previous attempts to prove this fact using Lyapunov functions had no success.
Instead, this paper takes a completely different approach, specific to antithetic configurations, in which the key idea is to think of the controller as providing a ``persistently negative feedback'' which acts far away from the equilibrium in such a way so as to keep trajectories from diverging.
This new approach, although tailored to the antithetic controller, might be useful in other applications as well.
\end{abstract}


\section{Introduction}\label{sec:intro duction}
This paper studies the following nonlinear dynamical system, which represents the closed-loop interconnection of an \emph{antithetic integral controller} (described by the variables $z_1\in\mathbb{R}_+$ and $z_2\in\mathbb{R}_+$) with a general positive linear system (described by the state $x\in\mathbb{R}_+^n$ and output $y\in\mathbb{R}_+$):
\begin{equation}\label{eq:gen_dynamics}
    \begin{aligned}
    \dot z_1 &= \alpha_1-\alpha_2 z_1 z_2\\[.25em]
    \dot z_2 &= \alpha_3 y-\alpha_4 z_1 z_2
    \end{aligned}
    \qquad \textrm{and} \qquad
    \begin{aligned}
    \dot x &= Ax + bz_1,\\
    y &= c^{\top} x.
    \end{aligned}
\end{equation}
See Figure~\ref{fig:diagram_general}.
\begin{figure}
    \centering
    \includegraphics[width=0.5\linewidth]{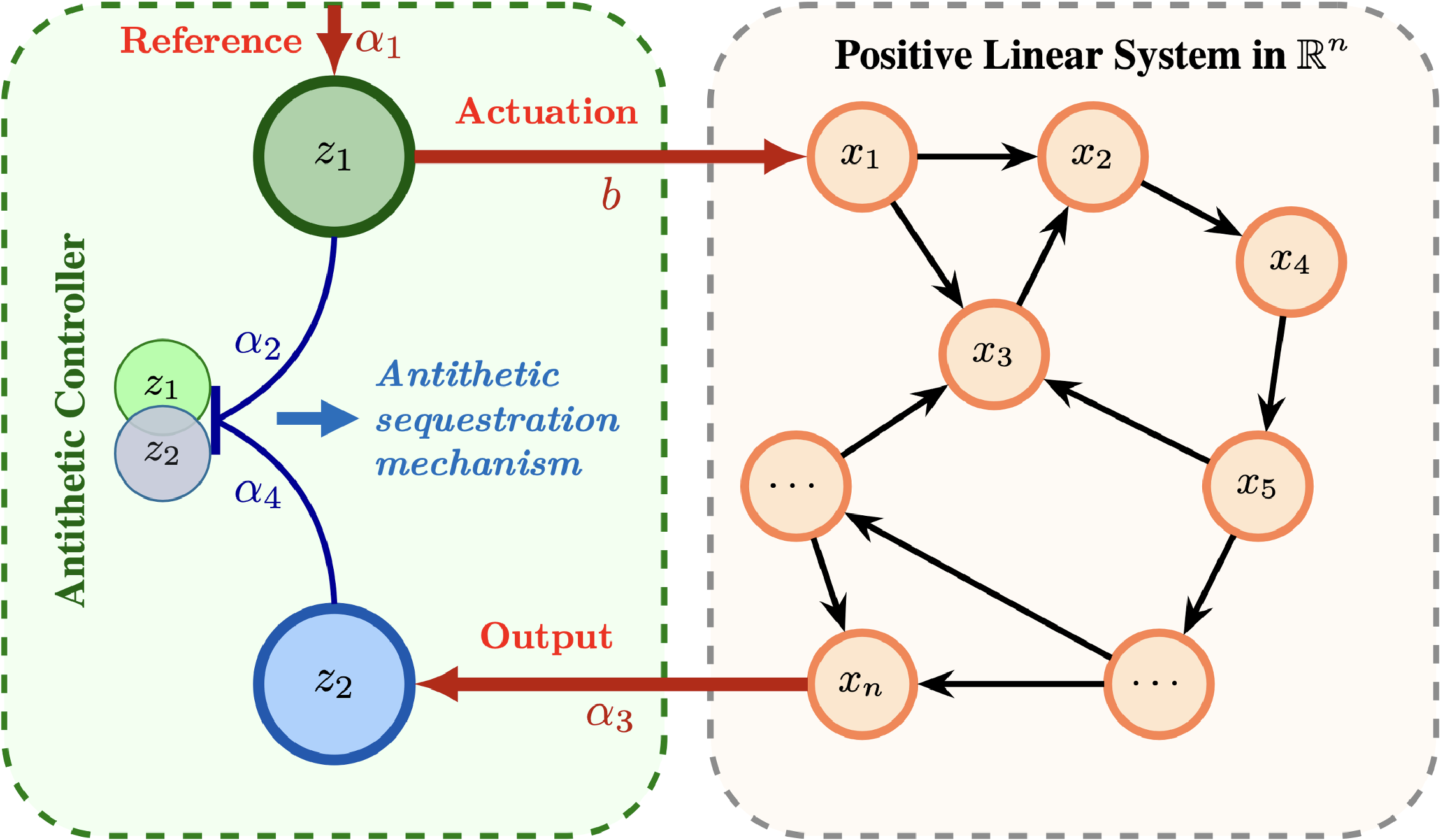}
    \caption{Closed-loop system~\eqref{eq:gen_dynamics}. The antithetic
    controller ($z_1$, $z_2$) on the left is interconnected with the positive
    linear plant on the right, where $A\in\mathbb{R}^{n\times n}$ is Hurwitz
    and Metzler, $b=[b_1,\dots,b_n]^\top,c=[c_1,\dots,c_n]^\top\in\mathbb{R}^n_+$, and the DC gain
    $G:=-c^\top A^{-1}b>0$. The signal $\alpha_1$ denotes the constant production
    rate of $z_1$. The red arrows indicate the main interconnection signals:
    $z_1$ drives the plant input with gain $b$, and the plant output
    $y = c^\top x$ feeds back to $z_2$ with gain $\alpha_3$. The vertical
    bar with rate $(\alpha_2,\alpha_4)$ represents the sequestration reaction between $z_1$
    and $z_2$, implementing integral action. Solid arrows inside the plant denote positive interactions between 
    state variables.}
    \label{fig:diagram_general}
\end{figure}
We assume that the constants $\alpha_i$, for $i=1,\dots,4$, are all positive, the matrix $A\in\mathbb{R}^{n\times n}$ 
is Hurwitz and Metzler, $b,c\in\mathbb{R}^n_+$, and the DC gain $G:=-c^\top A^{-1}b>0$ is positive. 
Our main result will show that every solution of this system is bounded.

Since its original formulation by Briat, Gupta, and Khammash \cite{AFC2016}, the antithetic controller has attracted sustained attention in both theoretical and experimental systems biology.
It is well known from linear systems theory, as well as from several classes of nonlinear systems, that integral feedback, or more generally the presence of an ``internal model'' of admissible exogenous signals, is essential for robust regulation. This principle also plays a central role in biological processes at both the cellular and organismal levels~\cite{rev_internal_model_2022}. Motivated by these considerations, there has been considerable recent effort devoted to 
implementing integral feedback mechanisms in synthetic biology (see~\cite{briat2025structuralstabilitypropertiesantithetic} and the references therein). A particularly influential approach, and the focus of this work, is based on sequestration or annihilation reactions between two molecular species ($z_1$ and $z_2$ in system~\eqref{eq:gen_dynamics}). This framework, introduced under the name antithetic control in~\cite{AFC2016}, implements integral action through simple biochemical interactions. Furthermore, it was shown in~\cite{Khammash2019} that such controllers achieve robust perfect adaptation 
while exhibiting, in a precise sense, minimal structural complexity.
Several research groups have reported experimental realizations of this architecture. An \emph{in vivo} implementation described in~\cite{Khammash2019} uses the $\sigma^W$--RsiW regulatory pair in \emph{Bacillus subtilis} as the antithetic species. The implementation~\cite{agarwal2019naturecom} constructed and analyzed an \emph{in vitro} synthetic biomolecular integral controller that regulates the protein production rate of an output gene via sequestration between 
\emph{E.\ coli} $\sigma_{28}$ and its corresponding anti-factor. Also closely related is a proposed ``quasi-integral feedback'' mechanism to achieve adaptation under varying ribosome demand, based on sRNA regulation~\cite{Huang2018}.

It is easy to show that the closed-loop system~\eqref{eq:gen_dynamics} admits a 
unique equilibrium $(z_1^*, z_2^*, x^*, y^*)$ in the nonnegative orthant, given by
\begin{equation*}
    z_1^* = \frac{\alpha_1\alpha_4}{\alpha_2\alpha_3 G}, \qquad
    z_2^* = \frac{\alpha_3 G}{\alpha_4}, \qquad
    x^* = -A^{-1}b\, z_1^*, \qquad
    y^* = \frac{\alpha_1\alpha_4}{\alpha_2\alpha_3},
\end{equation*}
for all positive $\alpha_1,\dots,\alpha_4$ and any admissible $A$, $b$, $c$ (i.e., $A$ Hurwitz and Metzler, $b,c\in\mathbb{R}^n_+$, and $G>0$).

There has been substantial work on the global dynamics of such systems, particularly in the special case, described in detail in the Supplemental Materials of~\cite{AFC2016}, in which the linear subsystem is a cascade of two one-dimensional systems. (The problem is already interesting even in the simplest case of static feedback~\cite{olsman2019antithetic}.) A key insight is that these systems (for arbitrary cascades of one-dimensional sytems, and even for a larger class of tridiagonal-Jacobian systems) can be analyzed using the theory of strongly $2$-cooperative systems; see~\cite{2019biorxiv_margaliot_sontag,katz2025instability,Margaliot-CDC}. This theory implies that every bounded trajectory satisfies a strong Poincar\'e--Bendixson property, in the sense of~\cite{Eyal_k_posi}. Consequently, every bounded trajectory whose omega-limit set does not contain the equilibrium must approach a periodic orbit. Such orbits do indeed occur: for appropriate choices of parameters and system matrices, the system can exhibit sustained oscillations, as illustrated by the four-dimensional example in Section~\ref{sec:numerics}. However, the question of whether \textit{all} trajectories are bounded, for arbitrary initial conditions, was left open. The present work resolves this issue by proving boundedness in full generality.

The key feature of the system is that the damping term $\alpha_2 z_1 z_2$ in $\dot z_1 = \alpha_1 - \alpha_2 z_1 z_2$ is not directly available as a feedback signal, but is generated through the $x$-system being controlled (the ``plant'') driven by $z_1$ itself. The central idea is that sufficiently large excursions of $z_1$ cannot persist indefinitely: if $z_1$ remains above a threshold $L > L_0$ for a sufficiently large time $T \ge T_0$, the plant output $y$ builds up sufficiently to drive $z_2$ large enough so that the product $z_1 z_2$ exceeds $\alpha_1/\alpha_2$, forcing $\dot{z}_1(T_0) < 0$. 
Further analysis then shows this domination persists for all $t \in [T_0, T]$, from which global boundedness of all states follows.
This \emph{input-to-output persistency} exists generically for the $x$-subsystem and is what allows the proof of boundedness of the integral controller output.

Notice in particular that the proof strategy does not rely on Lyapunov functions, instead working directly with differential inequalities and comparison principles, tracing lower bounds through the plant dynamics via the step-response function $g(t)$. The boundedness mechanism admits a natural interpretation 
in which forward propagation through the plant is eventually counteracted by the delayed negative feedback through $z_2$, offering clear insight into how the feedback delay influences the overall system behavior.

Understanding boundedness and stability properties of nonlinear feedback systems is a fundamental problem in control theory. In many applications, the feedback mechanism that regulates a system does not act instantaneously, 
but is mediated through intermediate dynamics that introduce delays and attenuation. As a result, classical Lyapunov-based approaches may be difficult to apply directly, especially when the feedback enters the system 
in a nonlinear or multiplicative form. We believe that the techniques introduced here may be helpful in the analysis of many such systems.

The organization of this paper is as follows.
In Section~\ref{sec:general}, we prove the main boundedness result, with a couple of routine technical lemmas deferred to an appendix.
In Section~\ref{sec:4d}, we specialize to the motivating four-dimensional case, with the purpose of giving more explicit estimates of the compact region that contains any given trajectory, and provide two numerical examples in Section~\ref{sec:numerics}.
Finally, in Section~\ref{sec:conclusion}, we include some additional discussion, including possible extensions of the current result for the interconnection of monotone nonlinear systems and the antithetic control structure.

\section{Main boundedness results}\label{sec:general}

We denote the set of real numbers, real vectors of dimension $n$, and 
real matrices of dimension $n\times m$ by $\mathbb{R}$, $\mathbb{R}^n$, 
and $\mathbb{R}^{n\times m}$. The orders $>,<,\geq$, and $\leq$ are 
interpreted element-wise when applied to vectors and matrices. We also define $\mathbb{R}_+^{n\times m}:=\{A\in\mathbb{R}^{n\times m}\mid A\ge 0\}$; 
the sets $\mathbb{R}_+^{n}$ and $\mathbb{R}_+:=[0,\infty)$ are defined similarly. A 
matrix is called \emph{Metzler} if its off-diagonal entries are nonnegative, and Hurwitz if its eigenvalues have strictly negative real parts.

The main contribution of this paper is the following result, which 
establishes boundedness of all solutions of~\eqref{eq:gen_dynamics}.

\begin{theorem}\label{thm:main}
    Every solution of system~\eqref{eq:gen_dynamics} starting from an initial 
    condition in $\mathcal{C} := \mathbb{R}^2_+ \times \mathbb{R}^n_+$ exists for all $t\ge 0$ and is bounded.
\end{theorem}

We prove this theorem in steps. First, we remark forward completeness of the interconnected system \eqref{eq:gen_dynamics} in Lemma \ref{lem:GenLin-ExSol}. Then, Theorem~\ref{thm:GenLin-HasProperty} 
establishes the key decrease property: there exist thresholds $L_0,T_0>0$ such 
that whenever $z_1(t)\ge L>L_0$ for a duration $T\ge T_0$, the plant output 
$y$ builds up enough to drive $z_2$ large, forcing $z_1 z_2>\alpha_1/\alpha_2$ 
and hence $\dot z_1<0$. Next, this local decrease property is promoted to 
global boundedness of $z_1$ in Theorem~\ref{thm:z1_bounded}. Finally, 
boundedness of $x$ and $y$ follows from the Hurwitz stability of $A$ with 
bounded input $z_1$ (Lemma~\ref{lem:x_bounded}), and boundedness of 
$z_2$ follows from a Lyapunov-like argument via the auxiliary function 
$W(t):=z_2(t)+\nu^\top x(t)$ (Lemma~\ref{lem:z2_bounded}).

\subsection{Proof of Theorem~\ref{thm:main}}

We first establish forward invariance of $\mathcal{C}= \mathbb{R}^2_+ \times 
\mathbb{R}^n_+$. Since the vector field defining~\eqref{eq:gen_dynamics} is 
polynomial, hence smooth and locally Lipschitz, solutions exist uniquely on 
a maximal interval $[0,t_{\max})$ for any initial condition in $\mathcal{C}$.

\begin{lemma}\label{lem:GenLin-OctInv}
    Let $w(0):=[z_1(0);z_2(0);x(0)]$ be any point in the nonnegative orthant $\mathcal{C}= \mathbb{R}^2_+ \times \mathbb{R}^n_+$. Then the solution $w(t):=[z_1(t);z_2(t);x(t)]$ of~\eqref{eq:gen_dynamics} initialized at $w(0)$
    satisfies $w(t)\in\mathcal{C}$ for all $t\in[0,t_{\max})$.
\end{lemma}

\begin{proof}
    This is a standard consequence of the Bony--Brezis invariance theorem 
    applied to the nonnegative orthant. For completeness, 
    a detailed proof is provided in the Appendix~\ref{pr:lemma}.
\end{proof}

Next, we show that solutions of~\eqref{eq:gen_dynamics} are globally defined.

\begin{lemma}\label{lem:GenLin-ExSol}
    The closed-loop system~\eqref{eq:gen_dynamics} is forward complete, that is for any 
    initial condition $z_1(0),z_2(0)\ge 0$ and $x(0)\ge 0$, the solution exists 
    and is unique for all $t\ge 0$.
\end{lemma}

\begin{proof}
    Consider the maximal interval $[0,t_{\max})$. We will next show that for any finite $t_{\max}$, the solution of \eqref{eq:gen_dynamics} initialized in the positive orthant $\mathcal{C}$ will remain inside a compact set, which will satisfy the standard continuation theorem and contradict maximality of $t_{\max}$.

    Let $z_1(t),z_2(t),x(t)$ be the solution of the initial value problem defined by \eqref{eq:gen_dynamics} and initialized at arbitrary $z_1(0),z_2(0),x(0)\ge0$. Notice that $\dot z_1\leq \alpha_1$ and thus integrating for $t\in[0,t_{\max})$ results in $z_1(t)\leq z_1(0)+\alpha_1t_{\max}=:z_{1,\max}$. From there write
    \begin{align*}
        \|x(t)\| &= \|\e^{At}x(0)+\int_0^t\e^{As}bz_1(t-s)\,\mbox{d}s\|\\
        &\leq \max_{t\in[0,t_{\max}]}\|\e^{At}\|\left(\|x(0)\|+t_{\max}\|b\|z_{1,\max}\right)=:x_{\max}.
    \end{align*}
    Next, notice that $\|y(t)\| = \|c^\top x(t)\|\leq \|c\|\|x(t)\|\leq \|c\|x_{\max}=:y_{\max}$, and finally that $\dot z_2 = \alpha_3 y - \alpha_4z_1z_2\leq \alpha_3 y_{\max}$ implying that $z_2(t)\leq z_2(0)+\alpha_3t_{\max}y_{\max}=:z_{2,\max}$. 

    We can, then, define the compact set $M = \{[z_1;z_2;x]\in\mathcal{C}~|~z_1\leq z_{1,\max},~z_2\le z_{2,\max},~\|x\|\leq x_{\max}\}$ which contains the solution $[z_1(t);z_2(t);x(t)]$ for all $t\in[0,t_{\max})$. Hence the maximal solution is bounded on $[0,t_{\max})$, and since the vector field is locally Lipschitz, the standard continuation theorem implies that the solution can be extended beyond $t_{\max}$. Therefore $t_{\max}=\infty$; see, for instance, Proposition C.3.6 in \cite{mct}.
\end{proof}

The result above establishes existence of solutions for all time, however those solutions may still be unbounded. We will next focus on showing boundedness of solutions for \eqref{eq:gen_dynamics}, and the key point that allows us to do that is the characterization of the following property:

\begin{theorem}
    \label{thm:GenLin-HasProperty}
    For the closed loop system in \eqref{eq:gen_dynamics}, there exist constants $L_0>0$ and $T_0>0$ such that for every $L>L_0$ and $T\geq T_0$, if a solution satisfies $z_1(0)=L$ and $z_1(t)\geq L$ for all $t\in [0,T]$, then $\dot z_1(T)<0$.
\end{theorem}

\begin{proof}
    The proof will proceed as follows: first we show that if $z_1(t)$ is ``large'' for a sufficiently long time, then the output $y(t)$ will also become proportionally large persistently. We then show, by a similar argument, that if $y(t)$ is ``large'' for a sufficiently long time, then $z_2(t)$ will also become proportionally large. We finally show that if everything is large enough at some point in time, then the product $z_1(t)z_2(t)$ will be \emph{and remain} large enough so that $\dot z_1<0$ for all future time. From the notions of ``large'' and ``long enough time'' throughout the proof, natural definitions for $L_0$ and $T_0$ arise. In the proof we will first give these expressions and later show that they are the correct values for the desired guarantees.

    Let 
    \begin{equation*}
        g(t):=c^\top \int_0^t\e^{As}b\,\mbox{d}s=c^\top A^{-1}(\e^{At}-I)b.
    \end{equation*}
    Since $A$ is Metzler, $\e^{At}\geq 0$ (element-wise) for all $t\ge0$ (see, for example, \cite{farina2000positive}). This implies $g(t)$ is nondecreasing since $\dot g(t)=c^\top \e^{At}b\geq0$. Moreover, since $A$ is also Hurwitz, the following limit is well defined and finite
    \begin{equation*}
        \lim_{t\to\infty}g(t)=c^\top\int_0^\infty \e^{As}b\,\mbox{d}s=-c^\top A^{-1}b=G > 0.
    \end{equation*}
    Since $g(0)=0$, $g(t)\to G$ and $g(t)$ is continuous, 
    there exists a smallest time $\Delta>0$ such that $g(t)\ge G/2$ for all $t\ge\Delta$.
    We define a function
    $\tau:\mathbb{R}_+\to\mathbb{R}_+$, where for each $\ell\in\mathbb{R}_+$, $\tau(\ell)$ is the unique positive solution of
    \begin{equation}\label{eq:T0_def}
        \tau(\ell) = \Delta+\frac{\ln(2)}{\alpha_4(\ell+\alpha_1\tau(\ell))},
    \end{equation}
    which is guaranteed to exist and be unique by Proposition~\ref{prop:T_fixedpoint} in Appendix~\ref{sec:technical}. We also point out here that $\tau(\ell)$ is continuous on $\mathbb{R}_+$, strictly decreasing in $\ell$ and differentiable on $(0,\infty)$ 
    (with the corresponding right derivative at $\ell=0$, see Proposition~\ref{prop:T_fixedpoint}(ii));
    in particular $\tau'(\ell) < 0$ for every $\ell>0$.
    Next, we define $L_0$ as the unique positive solution $\ell$ of
    \begin{equation}\label{eq:L0_def}
        \frac{\alpha_3 G\ell^2}{4\alpha_4(\ell+\alpha_1\tau(\ell))}=\frac{\alpha_1}{\alpha_2},
    \end{equation}
    whose existence and uniqueness follow from Proposition~\ref{prop:Lell4_threshold} in Appendix~\ref{sec:technical}. 
    With this definition of $L_0$, we finally define $T_0:=\tau(L_0)$. 

    Under the assumption that $z_1(t)\ge L>L_0$ for all $t\in[0,T]$ with 
    $T\ge T_0$, we compute
    \begin{align*}
        y(t) &= c^\top x(t)\\
        &= c^\top \e^{At}x(0)+\int_0^t c^\top\e^{As}b\,z_1(t-s)\,\mathrm{d}s \\
        &\ge g(t)L,
    \end{align*}
    where we used that $z_1(t-s)\ge L$ and that $c^\top \e^{At}x(0)\geq0$, which follows from    
    $\e^{At}\ge 0$, $x(0)\ge 0$,  and $b,c\in\mathbb{R}_+^n$. 
    Since $T\ge T_0=\tau(L_0)>\tau(L)>\Delta$ by 
    Proposition~\ref{prop:T_fixedpoint}(iii), the interval $[\Delta,T]$ is 
    nonempty, and thus
    \begin{equation}\label{eq:y_low}
        y(t) \ge g(t)L \ge \frac{GL}{2} =:\rho_y(L)
        \qquad\text{for all } t\in[\Delta,T].
    \end{equation}
    We now lower bound $z_2$ for $t\le T_0$ by first establishing an upper bound for $z_1(t)$. 
    Define 
    \[
        U(\ell,t):=\ell+\alpha_1t.
    \]
    Since $\dot z_1 = \alpha_1 - \alpha_2 z_1 z_2 \le \alpha_1$ and $z_1(0) = L$, integration gives
    \[
        z_1(t) \le L + \alpha_1 T_0 = U(L,T_0)
        \qquad\text{for all } t\in[0,T_0].
    \]
    Using $y(t)\ge \rho_y(L)$ for $t\ge\Delta$ 
    and $z_1(t)\le U(L,T_0)$ for $t\le T_0$, we obtain 
    \begin{align*}
        \dot z_2 = \alpha_3 y - \alpha_4 z_1 z_2 
        \ge \alpha_3 \rho_y(L) - \alpha_4 U(L,T_0)\,z_2
        \qquad\text{for all } t\in[\Delta,T_0].
    \end{align*}
    By the comparison principle (Gronwall's Lemma), we have $z_2(t)\ge v(t)$ where $v$ solves
    $\dot v = \alpha_3\rho_y(L) - \alpha_4 U(L,T_0)v$, $v(\Delta)=z_2(\Delta)\ge 0$,
    with explicit solution
    \begin{equation*}
        v(t) = \frac{\alpha_3 GL}{2\alpha_4 U(L,T_0)}
        \Bigl[1-e^{-\alpha_4 U(L,T_0)(t-\Delta)}\Bigr]
        + z_2(\Delta)\,e^{-\alpha_4 U(L,T_0)(t-\Delta)}.
    \end{equation*}
    Since $z_2(\Delta)\ge 0$, dropping the initial condition term yields
    \begin{equation}\label{eq:z2_low}
        z_2(t) \ge \frac{\alpha_3 GL}{2\alpha_4 U(L,T_0)}
        \Bigl[1-e^{-\alpha_4 U(L,T_0)(t-\Delta)}\Bigr]
        \qquad\text{for all } t\in[\Delta,T_0].
    \end{equation}
    Define
    \[
        \delta(\ell,t) \,:=\; \frac{\ln 2}{\alpha_4 U(\ell,t)}.
    \]
    Since $\tau(\ell)$ is strictly decreasing by Proposition~\ref{prop:T_fixedpoint}(ii) 
    and $U(\ell,t)$ is increasing in $t$, it follows that $\delta(\ell,t)$ is decreasing in $t$.
    From that we have
    \begin{equation*}
        T_0 = \tau(L_0) > \tau(L) = \Delta + \delta(L,\tau(L)) 
        > \Delta + \delta(L,T_0),
    \end{equation*}
    so the interval $[\Delta+\delta(L,T_0),\,T_0]$ is nonempty. 
    Evaluating~\eqref{eq:z2_low} at time $\Delta+\delta(L,T_0)$, 
    the exponential term equals $1/2$ by definition of $\delta$, giving
    \begin{equation}\label{eq:z2_threshold}
        z_2(t) \ge v(t)
        \ge \frac{\alpha_3 GL}{4\alpha_4(L+\alpha_1 T_0)} 
        =: \rho_z(L,T_0),
        \qquad\text{for all } t\in[\Delta+\delta(L,T_0),\,T_0].
    \end{equation}
    
    We have now established that for all $t\in[\Delta+\delta(L,T_0),\,T_0]$:
    $z_1(t)\ge L$, $y(t)\ge \rho_y(L)$, and $z_2(t)\ge \rho_z(L,T_0)$.
    In particular, from~\eqref{eq:z2_threshold} and the assumption $z_1(T_0)\ge L$, we obtain
    \begin{subequations}
    \begin{equation*}
        z_1(T_0)z_2(T_0) \ge L\rho_z(L,T_0) 
        = \frac{\alpha_3 GL^2}{4\alpha_4(L+\alpha_1 T_0)}.
    \end{equation*}
    By the choice of $L_0$ and since the map 
    $\ell\mapsto\ell\rho_z(\ell,T_0)$ is strictly increasing 
    on $[0,\infty)$, it follows 
    that for every $L>L_0$,
    \begin{equation*}
        L\rho_z(L,T_0) > L_0\rho_z(L_0,T_0) 
        = \frac{\alpha_1}{\alpha_2},
    \end{equation*}
    where the last equality holds by definition of $L_0$ in \eqref{eq:L0_def}. Therefore $z_1(T_0)z_2(T_0) > \alpha_1/\alpha_2$ and hence
    \begin{equation*}
        \dot z_1(T_0) = \alpha_1 - \alpha_2 z_1(T_0)z_2(T_0) < 0.
    \end{equation*}
    \end{subequations}
    However, as $T$ increases beyond $T_0$, the bound $\rho_z(L,T)$ in \eqref{eq:z2_threshold}
    decreases in $T$, so this estimate does not directly extend beyond $T_0$.
    To overcome this, we introduce the product
    \begin{equation*}
        p(t):=z_1(t)z_2(t),
    \end{equation*}
    and point out that $z_1(T_0)z_2(T_0)>\alpha_1/\alpha_2$ from the results above. By Lemma~\ref{lem:p_above_theta}, once the threshold
    $\theta=\alpha_1/\alpha_2$ is exceeded at time $T_0$, it cannot be crossed
    again from above. Consequently, $p(t)>\theta$ for all $t\in[T_0,T]$. Therefore,
    \begin{equation*}
        \dot z_1(t)
        =\alpha_1-\alpha_2p(t)
        <\alpha_1-\alpha_2\theta
        =0,
        \qquad\text{for all }t\in[T_0,T].
    \end{equation*}
    In particular, $\dot z_1(T)<0$, completing the proof.
\end{proof}

The following two results leverage the results in Theorem \ref{thm:GenLin-HasProperty} to prove boundedness of $z_1(t)$.

\begin{lemma}\label{lem:z1_overshoot}
    Let $L_0$ and $T_0$ be defined as in Theorem~\ref{thm:GenLin-HasProperty}. For 
    every $L>L_0$, if a solution satisfies $z_1(0)=L$, then
    $z_1(t) \le L + \alpha_1 T_0$ for all $t\ge 0$.
\end{lemma}

\begin{proof}
    We prove the claim by contradiction. Suppose that there exists $t^*>0$ such that 
    $z_1(t^*)>L+\alpha_1 T_0$. Define
    \begin{equation*}
        t_1 := \min\{t\in(0,t^*) : z_1(t)=L+\alpha_1 T_0\},
    \end{equation*}
    which exists by the intermediate value theorem since $z_1(0)=L<L+\alpha_1T_0$ 
    and $z_1(t^*)>L+\alpha_1T_0$. By definition, $z_1(t_1)=L+\alpha_1T_0$ and 
    $z_1(t)<L+\alpha_1T_0$ for all $t\in[0,t_1)$, which implies by Lemma~\ref{lem:first_hitting} that
    \begin{equation*}
        \dot z_1(t_1)\ge 0.
    \end{equation*}
    Next, define $t_2 := \max\{t\in[0,t_1] : z_1(t)=L\}$.    
    Note that the set is nonempty since $z_1(0)=L$, and that
    $t_2<t_1$ because $z_1(t_1)=L+\alpha_1T_0>L$
    . By definition, $z_1(t_2)=L$ and 
    $z_1(t)>L$ for all $t\in(t_2,t_1]$. Since $\dot z_1\le\alpha_1$, 
    integration on $[t_2,t_1]$ gives
    \begin{equation*}
        L+\alpha_1 T_0 = z_1(t_1) \le z_1(t_2)+\alpha_1(t_1-t_2)
        = L + \alpha_1(t_1-t_2),
    \end{equation*}
    hence $t_1-t_2\ge T_0$. Since~\eqref{eq:gen_dynamics} is time-invariant, 
    applying Theorem~\ref{thm:GenLin-HasProperty} on the shifted interval 
    $[t_2,t_1]$ with duration $T=t_1-t_2\ge T_0$ gives
    \begin{equation*}
        \dot z_1(t_1)<0,
    \end{equation*}
    contradicting $\dot z_1(t_1)\ge 0$. 
    Therefore no such $t^*$ exists, and $z_1(t)\le L+\alpha_1 T_0$ for all $t\ge 0$.
\end{proof}

\begin{theorem}\label{thm:z1_bounded}
    Let $L_0$ and $T_0$ be as in Theorem~\ref{thm:GenLin-HasProperty}. Every 
    solution of~\eqref{eq:gen_dynamics} satisfies
    \begin{equation*}
        z_1(t) \le \max\{z_1(0),\,L_0\} + \alpha_1 T_0 \eqqcolon M_{z_1} 
        \qquad\text{for all } t\ge 0.
    \end{equation*}
    In particular, $z_1$ is bounded on $[0,\infty)$.
\end{theorem}

\begin{proof}
    We consider two cases.
    
    \textbf{Case 1: $z_1(0)>L_0$.} Set $L:=z_1(0)>L_0$. 
    Lemma~\ref{lem:z1_overshoot} applies directly and gives
    \begin{equation*}
        z_1(t) \le z_1(0)+\alpha_1 T_0 \qquad\text{for all } t\ge 0.
    \end{equation*}
    \textbf{Case 2: $z_1(0)\le L_0$.} Fix any $\varepsilon>0$ and set 
    $L:=L_0+\varepsilon>L_0$. We claim that $z_1(t)\leq L+\alpha_1T_0$ for all $t\ge 0$.
    Suppose for contradiction that there exists 
    $t_1>0$ such that $z_1(t_1)>L+\alpha_1 T_0$. Since $z_1(0)\le L_0<L$ and $z_1(t_1)>L$,
    the intermediate value theorem gives that there is a well-defined $t_0 := \max\{t\in[0,t_1] : z_1(t)=L\}$. 
    
    Note that $z_1(t_0)=L$ and $z_1(t)>L$ for all $t\in(t_0,t_1]$. 
    Applying Lemma~\ref{lem:z1_overshoot} starting at time \(t_0\) onward with level \(L=L_0+\varepsilon\) gives
    \begin{equation*}
        z_1(t) \le L+\alpha_1 T_0 \qquad\text{for all } t\ge t_0,
    \end{equation*}
    which contradicts $z_1(t_1)>L+\alpha_1 T_0$. Since the bound 
    $z_1(t)\le L_0+\varepsilon+\alpha_1 T_0$ holds for every $\varepsilon>0$, 
    letting $\varepsilon\to 0^+$ gives $z_1(t)\le L_0+\alpha_1 T_0$ for all 
    $t\ge 0$.
    
    Combining both cases yields $z_1(t)\le\max\{z_1(0),L_0\}+\alpha_1 T_0
    =:M_{z_1}$ for all $t\ge 0$.
\end{proof}

From boundedness of $z_1(t)$, boundedness of $x(t)$ and $y(t)$ follow 
immediately: since $A$ is Hurwitz, the linear subsystem $\dot x = Ax + bz_1$ 
is bounded-input bounded-state stable (BIBS), and a bounded input $z_1$ therefore 
produces a bounded state $x$ and output $y = c^\top x$.
This is a routine textbook result (see e.g.\ Proposition 7.3.3 in~\cite{mct}), but it is useful to write explicitly the bounds, so we provide details next.

\begin{lemma}\label{lem:x_bounded}
    Let $z_1(t)\le M_{z_1}$ for all $t\ge 0$. Then there exist constants 
    $M_x, M_y>0$ such that $\|x(t)\| \le M_x$ and $y(t) \le M_y$ for all $t\ge 0$.
    In particular, $x$ and $y$ are bounded on $[0,\infty)$.
\end{lemma}

\begin{proof}
    The explicit solution of $\dot x = Ax + bz_1$ is given by
    \begin{equation*}
        x(t) = \e^{At}x(0) + \int_0^t \e^{A(t-s)}\,bz_1(s)\,\mbox{d}s.
    \end{equation*}
    Since $A$ is Hurwitz, there exist constants $m \ge 1$ and $\lambda > 0$ 
    such that $\|e^{At}\| \le m e^{-\lambda t}$ for all $t \ge 0$. Using the bound
    $z_1(s) \le M_{z_1}$, we obtain
    \begin{align*}
        \|x(t)\| 
        &\le \|\e^{At}\|\|x(0)\| 
         + \int_0^t \|\e^{A(t-s)}\|\|b\|\,z_1(s)\,\mbox{d}s \\
        &\le m\e^{-\lambda t}\|x(0)\| 
         + mM_{z_1}\|b\|\int_0^t \e^{-\lambda(t-s)}\,\mbox{d}s \\
        &= m\e^{-\lambda t}\|x(0)\| 
         + \frac{mM_{z_1}\|b\|}{\lambda}\bigl(1-\e^{-\lambda t}\bigr)\\
        &\le m\|x(0)\| + \frac{mM_{z_1}\|b\|}{\lambda}\eqqcolon M_x 
        \qquad\text{for all } t\ge 0,
    \end{align*}
    Hence $x$ is bounded on $[0,\infty)$. Since $y(t)=c^\top x(t)$ and $c\ge 0$, it follows 
    immediately that
    \begin{equation*}
        y(t) \le \|c\| M_x \eqqcolon M_y \qquad\text{for all } t \ge 0,
    \end{equation*}
    so $y$ is bounded on $[0,\infty)$.
\end{proof}

From boundedness of $x(t)$ and $y(t)$, boundedness of $z_2(t)$ follows 
via a Lyapunov-like argument using the auxiliary function 
$W(t):=z_2(t)+\nu^\top x(t)$.

\begin{lemma}\label{lem:z2_bounded}
    Let $z_1(t)\le M_{z_1}$, $\|x(t)\|\le M_x$, and $y(t)\le M_y$ for all 
    $t\ge 0$. Then there exists a constant $M_{z_2}>0$ such that $z_2(t) \le M_{z_2}$ for all $t\ge 0.$
    In particular, $z_2$ is bounded on $[0,\infty)$.
\end{lemma}

\begin{proof}
Define the auxiliary function
\begin{equation*}
    W(t) := z_2(t) + \nu^\top x(t), 
    \qquad \nu^\top := -\alpha_3 c^\top A^{-1} \ge 0
    \footnote{Since \(A\) is Hurwitz and Metzler, it is a nonsingular \(M\)-matrix. Hence $ -A^{-1}\ge 0$
    elementwise. Since $c\ge 0$ and $\alpha_3>0$, it follows that $\nu^\top=-\alpha_3 c^\top A^{-1}=\alpha_3 c^\top(-A^{-1})\ge 0.$ Thus \(\nu\ge 0\).},
\end{equation*}
where $\nu\in\mathbb{R}^n$ is well defined since $A$ is invertible. 
Differentiating and substituting the dynamics~\eqref{eq:gen_dynamics},
\begin{align*}
    \dot W 
    &= \dot z_2 + \nu^\top \dot x \\
    &= \bigl(\alpha_3 y - \alpha_4 z_1 z_2\bigr) 
       + \nu^\top\bigl(Ax + bz_1\bigr) \\
    &= \alpha_3 c^\top x - \alpha_4 z_1 z_2 
       + (\nu^\top A)x + (\nu^\top b)\,z_1.
\end{align*}
By the choice of $\nu^\top = -\alpha_3 c^\top A^{-1}$, the $x$-terms 
cancel: $\nu^\top A = -\alpha_3 c^\top$. Moreover, $\nu^\top b = -\alpha_3 c^\top A^{-1}b = \alpha_3 G,$
where $G:=-c^\top A^{-1}b>0$ is the DC gain. Therefore
\begin{equation}\label{eq:Wdot_general}
    \dot W = \alpha_3 G\,z_1 - \alpha_4 z_1 z_2 
           = \alpha_4 z_1\!\left(K^* - z_2\right),
\end{equation}
where $K^*:=\alpha_3 G/\alpha_4 = z_2^*$ is the equilibrium value of $z_2$. We now show that $W(t)$ is bounded, which in turn yields boundedness of $z_2$.

Since $\nu\ge 0$ and $x(t)\ge 0$, we have $\nu^\top x(t)\ge 0$, and therefore
\begin{equation*}
    z_2(t) = W(t) - \nu^\top x(t) \le W(t).
\end{equation*}

Set $\gamma := K^*+\|\nu\|M_x$. We claim that $W(t)\le\max\{W(0),\gamma\}$ 
for all $t\ge 0$. We consider two cases.

\textbf{Case 1: $W(0)<\gamma$.} Suppose for contradiction that $W(t_1)>\gamma$ 
for some $t_1>0$. Define
\[
    t_0:=\sup\{t\in[0,t_1]: W(t)\le\gamma\}.
\]
Since $W(0)<\gamma$ and $W(t_1)>\gamma$, we have $t_0\in[0,t_1)$.
By continuity, $W(t_0)=\gamma$, and by definition of $t_0$, $W(t)>\gamma$ for all $t\in(t_0,t_1]$. By Lemma~\ref{lem:x_bounded}, $\|x(t)\|\le M_x$ for all $t\ge0$.
Notice that, by the Cauchy--Schwarz inequality,
\begin{equation*}
    \nu^\top x(t)\le \|\nu\|\,\|x(t)\|\le \|\nu\|M_x
    \qquad\text{for all } t\ge0.
\end{equation*}
Hence, for all $t\in[t_0,t_1]$,
\[
    z_2(t)=W(t)-\nu^\top x(t)
    \ge W(t)-\|\nu\|M_x
    \ge\gamma-\|\nu\|M_x
    =K^*.
\]
Thus
\[
    \dot W(t)=\alpha_4z_1(t)(K^*-z_2(t))\le0
    \qquad\text{for all }t\in[t_0,t_1].
\]
Therefore $W$ is nonincreasing on $[t_0,t_1]$, and so
\[
    W(t_1)\le W(t_0)=\gamma,
\]
contradicting $W(t_1)>\gamma$. Hence $W(t)\le\gamma$ for all $t\ge0$.

\textbf{Case 2: $W(0)\ge\gamma$.} 
Suppose for contradiction that 
$W(t_1)>W(0)$ for some $t_1>0$. Since $W(t_1)>W(0)\ge\gamma$, the set
\begin{equation*}
    \Omega:=\{t\in[0,t_1]\mid W(t)>\gamma\}
\end{equation*}
is nonempty. Let $I$ be the connected component of $\Omega$ containing $t_1$.
Then there exists $t_0\in[0,t_1)$ such that
\begin{equation*}
    I=(t_0,t_1]
    \qquad\text{or}\qquad
    I=[t_0,t_1],
\end{equation*}
with $W(t)>\gamma$ for all $t\in I$. Moreover, either $W(t_0)=\gamma$ or $t_0=0$.
For all $t\in I$, we have
\begin{equation*}
    z_2(t) = W(t)-\nu^\top x(t) 
    \ge W(t)-\|\nu\|M_x > K^*.
\end{equation*}
Therefore, using~\eqref{eq:Wdot_general},
\begin{equation*}
    \dot W(t) = \alpha_4 z_1(t)\bigl(K^*-z_2(t)\bigr)\le 0
    \qquad\text{for all }t\in I.
\end{equation*}
Hence $W$ is nonincreasing on $I$.

If $t_0=0$, then, if necessary taking the limit from the right endpoint of the component,
\begin{equation*}
    W(t_1)\le W(0),
\end{equation*}
contradicting the choice of $t_1$. If $t_0>0$, then $W(t_0)=\gamma\le W(0)$, and again
taking the right limit if $I=(t_0, t_1]$,
\begin{equation*}
    W(t_1)\le W(t_0)=\gamma\le W(0),
\end{equation*}
contradicting $W(t_1)>W(0)$. Therefore $W(t)\le W(0)$ for all $t\ge0$ in this case.

Combining both cases, $W(t)\le\max\{W(0),\gamma\}$ for all $t\ge 0$.
Consequently, since $\nu\ge0$ and $x(t)\ge0$,
\begin{equation*}
    z_2(t) \le W(t)-\nu^\top x(t) \le W(t) \le \max\{W(0),\gamma\}\eqqcolon M_{z_2}
    \qquad\text{for all } t\ge 0.
\end{equation*}
Thus $z_2$ is bounded on $[0,\infty)$.
\end{proof}

We can now conclude the proof of Theorem~\ref{thm:main}. Lemma~\ref{lem:GenLin-ExSol} gives existence for all $t\ge0$. Theorem~\ref{thm:z1_bounded} gives boundedness of $z_1$. Lemma~\ref{lem:x_bounded} then gives boundedness of $x$ and $y$, and Lemma~\ref{lem:z2_bounded} gives boundedness of $z_2$. Hence the full solution $(z_1,z_2,x)$ is bounded on $[0,\infty)$.

\section{Explicit boundedness analysis in $\mathbb{R}^4$}\label{sec:4d}

As a concrete and biologically motivated instance of~\eqref{eq:gen_dynamics}, 
originally studied in~\cite{AFC2016}, we consider the four-dimensional system 
obtained by taking $n=2$, where the antithetic controller states $(z_1, z_2)$ 
regulate the two-dimensional linear cascade $(x_1, x_2)$:
\begin{equation}\label{eq:ca_dynamics}
    \begin{aligned}
    \dot z_1 &= \alpha_1-\alpha_2 z_1 z_2\\
    \dot z_2 &= \alpha_3 x_2-\alpha_4 z_1 z_2
    \end{aligned}
    \qquad \textrm{and} \qquad
    \begin{aligned}
    \dot x_1 &= \beta_1 z_1-\beta_2 x_1\\
    \dot x_2 &= \beta_3 x_1-\beta_4 x_2,
    \end{aligned}
\end{equation}
with $\alpha_i,\beta_i>0$ for $i=1,\dots,4$ and $z_1(0),z_2(0),x_1(0),x_2(0)\ge 0$. Since this system 
has been widely studied in the literature, we treat it explicitly here. 
Moreover, the explicit cascade structure $z_1\to x_1\to x_2\to z_2$ allows 
us to derive bounds stage by stage, rather than through the step-response 
estimate $y\ge GL/2$ used in the general case, yielding fully 
explicit bounds $\overline{M}_{z_1}, \overline{M}_{x_1}, \overline{M}_{x_2}, \overline{M}_{z_2}$ 
in terms of $\alpha_i$ and $\beta_i$ for $i=1,\dots,4$.

The correspondence of~\eqref{eq:ca_dynamics} with~\eqref{eq:gen_dynamics} 
is given by
\begin{equation*}
    A = \begin{bmatrix} -\beta_2 & 0 \\ \beta_3 & -\beta_4 \end{bmatrix}, \qquad
    b = \begin{bmatrix} \beta_1 \\ 0 \end{bmatrix}, \qquad
    c = \begin{bmatrix} 0 \\ 1 \end{bmatrix},
\end{equation*}
so that $x=[x_1,x_2]^\top$, $y=x_2$, and the DC gain is
$G = -c^\top A^{-1}b = \beta_1\beta_3/(\beta_2\beta_4)>0$.
Note that $\alpha_4$ appears only in $\dot z_2$ and does not affect $G$.

Since~\eqref{eq:ca_dynamics} is a special case of~\eqref{eq:gen_dynamics}, 
the main theorem of the general case applies directly:

\begin{corollary}\label{cor:4d_bounded}
    Every solution of system~\eqref{eq:ca_dynamics} starting from an initial 
    state in $\mathcal{C}:=\mathbb{R}^4_+$ is bounded.
\end{corollary}

The remainder of this section makes the boundedness result explicit in terms 
of $\alpha_i$ and $\beta_i$ for $i=1,\dots,4$.
Corollary~\ref{cor:z1_bounded_4d} gives an explicit bound $\overline{M}_{z_1}$ 
for $z_1$, and Corollaries~\ref{cor:x_bounded_4d}--\ref{cor:z2_bounded_4d} give explicit bounds 
$\overline{M}_{x_1}, \overline{M}_{x_2}, \overline{M}_{z_2}$ for $x_1, x_2, z_2$ respectively. 

\subsection{Proof of Corollary~\ref{cor:4d_bounded}}

We already know, from the general results, that $\mathcal{C}:=\mathbb{R}^4_+$ 
is forward invariant. 
Similarly, forward completeness 
follows directly from Lemma~\ref{lem:GenLin-ExSol}, since~\eqref{eq:ca_dynamics} 
is a special case of~\eqref{eq:gen_dynamics}.

Now we proceed to the key property of the proof.

\begin{corollary}\label{cor:z1_decrease}
    Consider system~\eqref{eq:ca_dynamics}. There exist constants $L_0>0$ and 
    $T_0>0$ such that for every $L>L_0$ and $T\ge T_0$, if a solution satisfies 
    $z_1(0)=L$ and $z_1(t)\ge L$ for all $t\in[0,T]$, then $\dot z_1(T)<0$.
\end{corollary}

\begin{proof}
The proof follows the same structure as Theorem~\ref{thm:GenLin-HasProperty}, 
but exploits the explicit cascade structure of~\eqref{eq:ca_dynamics}: lower 
bounds on $x_1$, $x_2$, and $z_2$ are derived stage by stage rather than 
via the step-response $g(t)$. We then use $L_0$ to show 
$z_1(T_0)z_2(T_0)>\alpha_1/\alpha_2$ for every $L>L_0$, forcing 
$\dot z_1(T_0)<0$, and verify this domination persists on $[T_0,T]$ while 
$z_1(t)\ge L$ for all $t\in[0,T]$.

Define $\Delta_1 := \ln 2/\beta_2$, $\Delta_2 := \ln 2/\beta_4$, and 
$\delta(\ell,t) := \ln 2/(\alpha_4 U(\ell,t))$ for $\ell\ge0$, $t\ge 0$, 
where $U(\ell,t):=\ell+\alpha_1 t$. For each $\ell\ge0$, define $\tau(\ell)$ 
as the unique positive solution of
\begin{equation*}
    \tau(\ell) = \Delta_1 + \Delta_2 + \delta(\ell,\tau(\ell)),
\end{equation*}
which is guaranteed to exist by Proposition~\ref{prop:T_fixedpoint}, 
taking $\Delta := \Delta_1+\Delta_2$ in that proposition. Then, fix $L_0$ as the unique positive solution of
\begin{equation*}
    \frac{\alpha_3 G\ell^2}{8\alpha_4(\ell+\alpha_1\tau(\ell))} = \frac{\alpha_1}{\alpha_2},
\end{equation*}
whose existence and uniqueness follow from Proposition~\ref{prop:Lell4_threshold},
where $G = \beta_1\beta_3/(\beta_2\beta_4)>0$ is the DC gain.
Once $L_0$ is fixed, we set $T_0:=\tau(L_0)$, which is a fixed positive 
constant. Since $\tau(\ell)$ is strictly decreasing 
(Proposition~\ref{prop:T_fixedpoint}(ii)), for every $L>L_0$ we have 
$\tau(L)<\tau(L_0)=T_0$, so that $T\ge T_0>\tau(L)$ whenever $T\ge T_0$. 
We will next show that the corollary holds for these choices of $L_0$ and $T_0$.

Let $L>L_0$, $T\ge T_0$ and suppose
\begin{equation}\label{eq:z1_low}
    z_1(t)\ge L, \qquad \text{for all } t\in[0,T].
\end{equation}
We derive lower bounds for $x_1(t)$, $x_2(t)$, and $z_2(t)$, each holding 
after a finite delay within $[0,T_0]\subset[0,T]$.

\begin{enumerate}[leftmargin=*]
    \item \textbf{Estimate of $x_1(t)$}. On $[0,T]$, using 
    $\dot x_1=\beta_1 z_1-\beta_2 x_1$ together with~\eqref{eq:z1_low}, 
    we obtain
    \begin{equation*}
        \dot x_1 \ge \beta_1 L - \beta_2 x_1.
    \end{equation*}
    Define $\dot w_1(t)=\beta_1 L-\beta_2 w_1(t)$ with $w_1(0)=x_1(0)$. 
    By the comparison principle, this leads to $x_1(t)\ge w_1(t)$ for all
    $t\in[0,T]$. Hence
    \begin{equation*}
        x_1(t)\ge \frac{\beta_1 L}{\beta_2}\Bigl[1-e^{-\beta_2 t}\Bigr] 
        + x_1(0)e^{-\beta_2 t}
        \ge \frac{\beta_1 L}{\beta_2}\Bigl[1-e^{-\beta_2 t}\Bigr].
    \end{equation*}
    Since $1-e^{-\beta_2 t}$ is increasing in $t$. By defining $\Delta_1\coloneqq(\ln 2)/\beta_2$, this yields
    \begin{equation}\label{eq:x1_low}
        x_1(t)\ge \frac{\beta_1 L}{\beta_2}\Bigl[1-e^{-\beta_2 t}\Bigr]\ge
        \frac{\beta_1 L}{\beta_2}\Bigl[1-e^{-\beta_2\Delta_1}\Bigr] 
        = \frac{\beta_1 L}{2\beta_2} \eqqcolon \bar\rho_1(L)
        \qquad\text{for all } t\in[\Delta_1,T].
    \end{equation}

    \item \textbf{Estimate of $x_2(t)$}. On the interval $[\Delta_1,T]$, 
    using $\dot x_2=\beta_3 x_1-\beta_4 x_2$ and~\eqref{eq:x1_low} 
    results in
    \begin{equation*}
        \dot x_2 \ge \beta_3\bar\rho_1(L) - \beta_4 x_2.
    \end{equation*}
    Define $\dot w_2(t)=\beta_3\bar\rho_1(L)-\beta_4 w_2(t)$ with 
    $w_2(\Delta_1)=x_2(\Delta_1)$. By the comparison principle, we have 
    $x_2(t)\ge w_2(t)$ for all $t\in[\Delta_1,T]$. Hence
    \begin{equation*}
        x_2(t) \ge \frac{\beta_3}{\beta_4}\bar\rho_1(L)
        \Bigl[1-e^{-\beta_4(t-\Delta_1)}\Bigr]
        + x_2(\Delta_1)e^{-\beta_4(t-\Delta_1)}
        \ge \frac{\beta_3}{\beta_4}\bar\rho_1(L)
        \Bigl[1-e^{-\beta_4(t-\Delta_1)}\Bigr].
    \end{equation*}
    Since $t-\Delta_1\ge0$, thus $1-e^{-\beta_4(t-\Delta_1)}$ is increasing in $t$. Defining $\Delta_2\coloneqq(\ln 2)/\beta_4$, this results in
    \begin{equation}\label{eq:x2_low}
        x_2(t)\ge\frac{\beta_3}{\beta_4}\bar\rho_1(L)
        \Bigl[1-e^{-\beta_4(t-\Delta_1)}\Bigr] \ge 
        \frac{\beta_3}{\beta_4}\bar\rho_1(L)
        \Bigl[1-e^{-\beta_4\Delta_2}\Bigr]
        =\frac{\beta_3}{2\beta_4}\bar\rho_1(L)
        \eqqcolon \bar\rho_2(L) 
        \qquad\text{for all } t\in[\Delta_1+\Delta_2,T].
    \end{equation}

    \item \textbf{Upper bound for $z_1$.}
    Before lower bounding $z_2$, we must first derive an upper bound for 
    $z_1$ in the interval. From the equation $\dot z_1=\alpha_1-\alpha_2 
    z_1 z_2$ and the nonnegativity $z_1,z_2\ge 0$, we obtain
    $\dot z_1(t)=\alpha_1-\alpha_2 z_1 z_2\le\alpha_1$ for all $t\ge 0.$
    Integrating on any interval $[0,t]$ yields
    \begin{equation*}
        z_1(t) \le z_1(0) + \alpha_1 t.
    \end{equation*}
    Using the assumption $z_1(t)\ge L$ for all $t\in[0,T_0]$ and 
    $z_1(0)=L$, then for every $t\in[0,T_0]$,
    \begin{equation}\label{eq:z1_window}
        z_1(t)\le L+\alpha_1 t
        \le L+\alpha_1T_0 = U(L,T_0).
    \end{equation}
    Thus, on any time interval of length $T_0$ on which $z_1$ remains 
    above $L$, it is also bounded from above by $U(L,T_0)$.

    \item \textbf{Estimate of $z_2(t)$}. On the interval 
    $[\Delta_1+\Delta_2,T_0]$, using $\dot z_2=\alpha_3 x_2-\alpha_4 
    z_1 z_2$, \eqref{eq:x2_low}, and~\eqref{eq:z1_window} yield
    \begin{equation*}
        \dot z_2(t) \ge \alpha_3\bar\rho_2(L) - \alpha_4 U(L,T_0)\,z_2(t).
    \end{equation*}
    Define $\dot v(t)=\alpha_3\bar\rho_2(L)-\alpha_4 U(L,T_0)v(t)$ with 
    $v(\Delta_1+\Delta_2)=z_2(\Delta_1+\Delta_2)$.
    By the comparison principle we have $z_2(t)\ge v(t)$ for all
    $t\in[\Delta_1+\Delta_2,T_0]$. Hence, with $\bar U\coloneqq U(L,T_0)$,
    \begin{equation*}
        \begin{aligned}
            z_2(t) &\ge \frac{\alpha_3}{\alpha_4\bar U}\bar\rho_2(L) 
            \Bigl[1-e^{-\alpha_4\bar U(t-(\Delta_1+\Delta_2))}\Bigr]
            + z_2(\Delta_1+\Delta_2)\,e^{-\alpha_4\bar U(t-(\Delta_1+\Delta_2))} \\
            &\ge \frac{\alpha_3}{\alpha_4\bar U}\bar\rho_2(L) 
            \Bigl[1-e^{-\alpha_4\bar U(t-(\Delta_1+\Delta_2))}\Bigr].
        \end{aligned}
    \end{equation*}
    Since $t-(\Delta_1+\Delta_2)\ge 0$, thus $1-e^{-\alpha_4\bar U(t-(\Delta_1+\Delta_2))}$ is increasing in $t$. Using $\delta(L,T_0)=(\ln 2)/(\alpha_4 U(L,T_0))$,
    \begin{equation}\label{eq:z2_low_4d}
        z_2(t) \ge \frac{\alpha_3}{2\alpha_4 U(L,T_0)}\bar\rho_2(L) 
        = \frac{\alpha_3 GL}{8\alpha_4(L+\alpha_1 T_0)} 
        \eqqcolon \bar\rho_z(L,T_0)
        \qquad\text{for all } t\in[\Delta_1+\Delta_2+\delta(L,T_0),T_0].
    \end{equation}
\end{enumerate}

Note that the bounds on $x_1$ and $x_2$ persist while $z_1(t)\ge L$, and 
hence are valid on $[\Delta_1,T]$ and $[\Delta_1+\Delta_2,T]$ respectively. 
For $z_2$, however, the damping coefficient $\alpha_4 z_1$ depends on $z_1$ 
itself, so we need the upper bound $z_1(t)\le U(L,T_0)$, which is only 
available on $[0,T_0]$. This is why the estimate for $z_2$ is restricted 
to $[0,T_0]$.

Unlike the general case, where $y\ge GL/2$ follows from $g(t)\ge G/2$, 
here the bounds on $x_1$, $x_2$, and $z_2$ are derived explicitly stage 
by stage through the cascade $z_1\to x_1\to x_2\to z_2$. 
From~\eqref{eq:z1_low} and~\eqref{eq:z2_low_4d}, 
$z_1(T_0)z_2(T_0)\ge L\bar\rho_z(L,T_0)$.
By the choice of $L_0$ and since $\ell\mapsto\ell\bar\rho_z(\ell,T_0)$ 
is strictly increasing in $\ell$, for every 
$L>L_0$,
\begin{equation*}
    L\bar\rho_z(L,T_0) > L_0\bar\rho_z(L_0,T_0) = \frac{\alpha_1}{\alpha_2},
\end{equation*}
where the last equality holds by definition of $L_0$. Thus
$z_1(T_0)z_2(T_0)>\alpha_1/\alpha_2$ and 
$\dot z_1(T_0)=\alpha_1-\alpha_2 z_1(T_0)z_2(T_0)<0$.

However, as $T$ increases beyond $T_0$, the bound $\bar\rho_z(L,T)$
decreases in $T$, so this estimate does not directly extend beyond $T_0$.
As in the proof of Theorem~\ref{thm:GenLin-HasProperty}, introduce the product
$p(t):=z_1(t)z_2(t)$ and $\theta\coloneqq\alpha_1/\alpha_2$.
We claim that $p(t)>\theta$ for all $t\in[T_0,T]$.
Indeed, the proof is identical to that of Lemma~\ref{lem:p_above_theta}, with
$\rho_z$ replaced by $\bar\rho_z$ and
$L^*:=2\alpha_4\theta/(\alpha_3G)$ replaced by
$L_*:=4\alpha_4\theta/(\alpha_3G)$.
    \footnote{The factor-of-two difference 
    between $L^*$ and $L_*$ reflects the difference in the lower bound on $y$: 
    in the general case, $y\ge GL/2$ follows from the step-response estimate 
    $g(t)\ge G/2$, while here the explicit cascade gives 
    $x_2\ge\bar\rho_2(L)=GL/4$, 
    which carries an additional factor of $1/2$ through the two cascade stages. 
    The remainder of the argument is unchanged: one shows 
    $L_*\bar\rho_{z}(L_*,\tau(L_*))<\theta$, which implies $L_0>L_*$, and 
    hence $\alpha_3\bar\rho_2(L_0)-\alpha_4\theta>0$, giving 
    $\dot p(t_0)>0$ and contradicting $\dot p(t_0)\le 0$.}
Therefore,
\begin{equation*}
    \dot z_1(t) = \alpha_1-\alpha_2 p(t) < \alpha_1-\alpha_2\theta = 0,
    \qquad\text{for all } t\in[T_0,T].
\end{equation*}
In particular, $\dot z_1(T)<0$, completing the proof.
\end{proof}

With Corollary~\ref{cor:z1_decrease} established, the global boundedness 
of $z_1$ follows by exactly the same argument as in the general case 
(Lemma~\ref{lem:z1_overshoot} and Theorem~\ref{thm:z1_bounded}).

\begin{corollary}\label{cor:z1_overshoot_4d}
    Let $L_0$ and $T_0$ be defined as in Corollary~\ref{cor:z1_decrease}. For 
    every $L>L_0$, if a solution satisfies $z_1(0)=L$, then
    $z_1(t) \le L + \alpha_1 T_0$ for all $t\ge 0$.
\end{corollary}

\begin{corollary}\label{cor:z1_bounded_4d}
    Let $L_0$ and $T_0$ be defined as in Corollary~\ref{cor:z1_decrease}. Every 
    solution of~\eqref{eq:ca_dynamics} satisfies
    \begin{equation*}
        z_1(t) \le \max\{z_1(0),\,L_0\} + \alpha_1 T_0 \eqqcolon \overline{M}_{z_1}
        \qquad\text{for all } t\ge 0.
    \end{equation*}
    In particular, $z_1$ is bounded on $[0,\infty)$.
\end{corollary}

Once $z_1$ is known to be bounded by $\overline{M}_{z_1}$, boundedness of $x_1$ 
and $x_2$ follows by direct comparison.

\begin{corollary}\label{cor:x_bounded_4d}
    Let $z_1(t)\le \overline{M}_{z_1}$ for all $t\ge 0$. There exist constants
    \begin{equation*}
        \max\left\{x_1(0),\,\frac{\beta_1}{\beta_2}\overline{M}_{z_1}\right\}\eqqcolon \overline{M}_{x_1}>0, \qquad
        \max\left\{x_2(0),\,\frac{\beta_3}{\beta_4}\overline{M}_{x_1}\right\}\eqqcolon \overline{M}_{x_2}>0,
    \end{equation*}
    such that
    $x_1(t)\le \overline{M}_{x_1}$ and $x_2(t)\le \overline{M}_{x_2}$ for all $t\ge 0$.
    In particular, $x_1$ and $x_2$ are bounded on $[0,\infty)$.
\end{corollary}

Once $x_1$ and $x_2$ are known to be bounded, boundedness of $z_2$ follows 
by the same Lyapunov-like argument as in the proof of Lemma~\ref{lem:z2_bounded}. 
The auxiliary function $W(t):=z_2(t)+\nu^\top x(t)$ takes the explicit form 
$W(t):=z_2(t)+cx_1(t)+dx_2(t)$, where $c:=\beta_3\alpha_3/(\beta_2\beta_4)$ 
and $d:=\alpha_3/\beta_4$ are the explicit values of 
$\nu^\top=-\alpha_3 c^\top A^{-1}$ for this system, and $K_*:=\alpha_3 G/\alpha_4$ 
coincides with $K^*$ in Lemma~\ref{lem:z2_bounded}, now with 
$G=\beta_1\beta_3/(\beta_2\beta_4)$ explicit.

\begin{corollary}\label{cor:z2_bounded_4d}
    Let $z_1(t)\le \overline{M}_{z_1}$, $x_1(t)\le \overline{M}_{x_1}$, and $x_2(t)\le \overline{M}_{x_2}$ for all 
    $t\ge 0$. There exists a constant $\overline{M}_{z_2}>0$ such that $z_2(t)\le \overline{M}_{z_2}$ 
    for all $t\ge 0$. In particular, $z_2$ is bounded on $[0,\infty)$.
\end{corollary}

\section{Numerical Example} \label{sec:numerics}

\subsection{Positive linear system in $\mathbb{R}^n$}\label{subsec:numerics_general}

We first illustrate the boundedness result for the general
system~\eqref{eq:gen_dynamics} with $n=12$. The system parameters are
$\alpha_1=1$, $\alpha_2=30$, $\alpha_3=1$, and $\alpha_4=30$. The matrix $A$ is constructed as a Hurwitz Metzler matrix by choosing
a nonnegative off-diagonal matrix $R\in\mathbb{R}^{12\times 12}$ with
randomly generated nonnegative entries and zero diagonal, and setting
\begin{equation*}
A = R-\diag\Bigl(1+\sum\nolimits_j R_{ij}\Bigr),
\end{equation*}
so that $A$ is Metzler and strictly diagonally dominant with negative
diagonal entries, hence Hurwitz. The vectors $b,c\in\mathbb{R}^{12}_+$
are chosen with random positive entries. The resulting theoretical constants are
\begin{equation*}
    \Delta_x \approx 0.6587,\quad
    G \approx 2.0684,\quad
    L_0 \approx 2.3237,\quad
    T_0 \approx 0.6664.
\end{equation*}
To illustrate the dependence on the initial condition, we simulate ten trajectories with
$z_2(0)=0,$ $x(0)=0,$ and $z_1(0)$ uniformly distributed in the interval $[0,L_0)$. Since
$z_1(0)\le L_0$ for all simulations, the theoretical bounds are identical across all trajectories and are given by
\begin{equation*}
    M_{z_1} \approx 2.9901,\qquad
    M_x \approx 7.4848,\qquad
    M_{z_2} \approx 13.2746.
\end{equation*}

Figure~\ref{fig:general_numerical_results} shows the simulated trajectories together with the corresponding theoretical bounds. Although the trajectories originate from different initial conditions, they all remain below the same bounds, illustrating the uniform nature of the estimates within the region $z_1(0)\le L_0$. The bound $M_{z_1}$ is relatively tight, while $M_x$ and $M_{z_2}$ are more conservative, as expected from the use of norm estimates and comparison arguments in the proof.

\begin{figure}[h!]
    \centering
    \subfloat[]{
    \includegraphics[width=0.2315\textwidth]{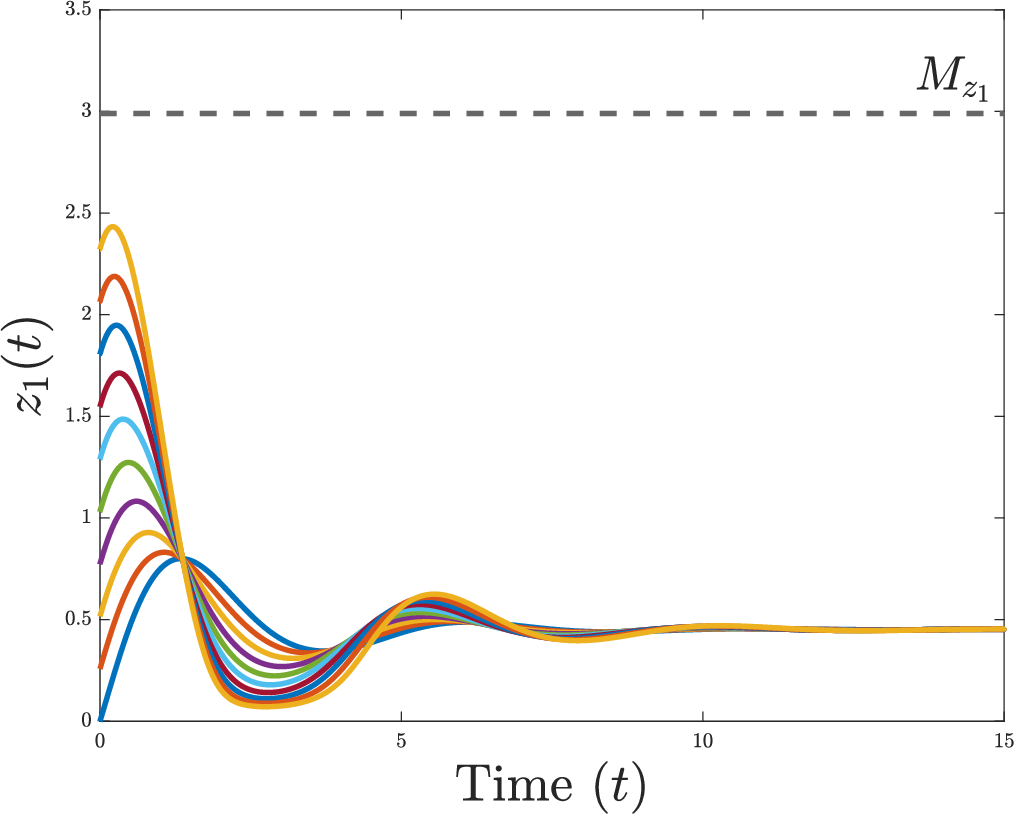}
    \label{fig:general_z1}
    }
    \subfloat[]{
    \includegraphics[width=0.23\textwidth]{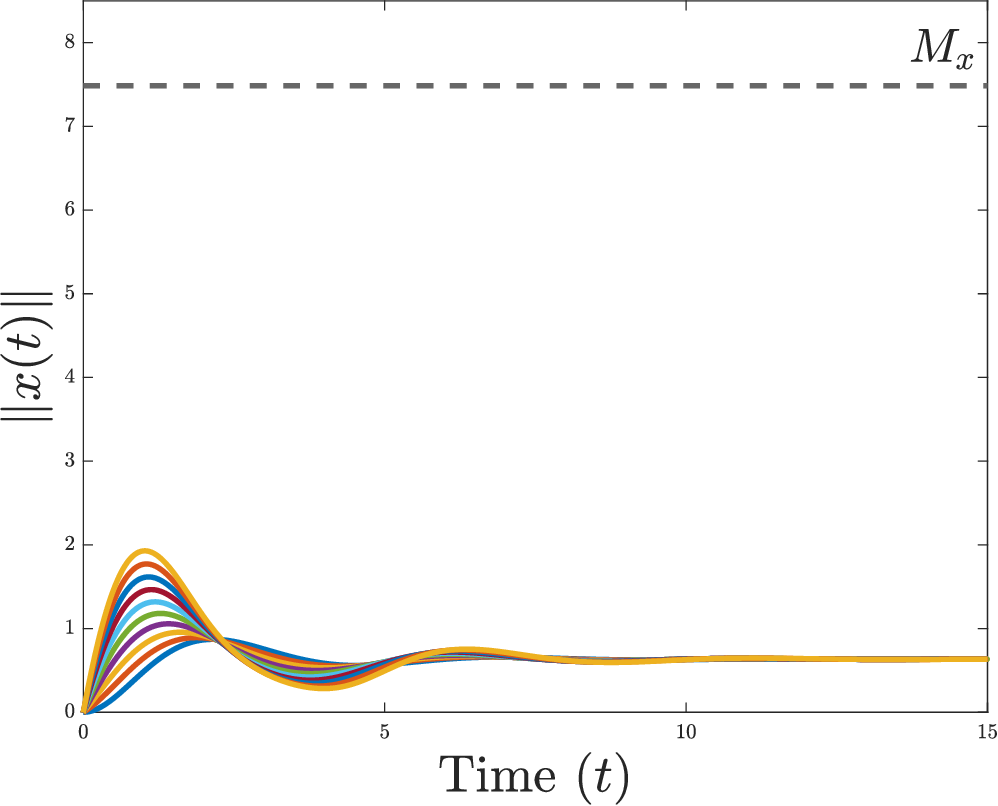}
    \label{fig:general_x}
    }
    \subfloat[]{
    \includegraphics[width=0.2315\textwidth]{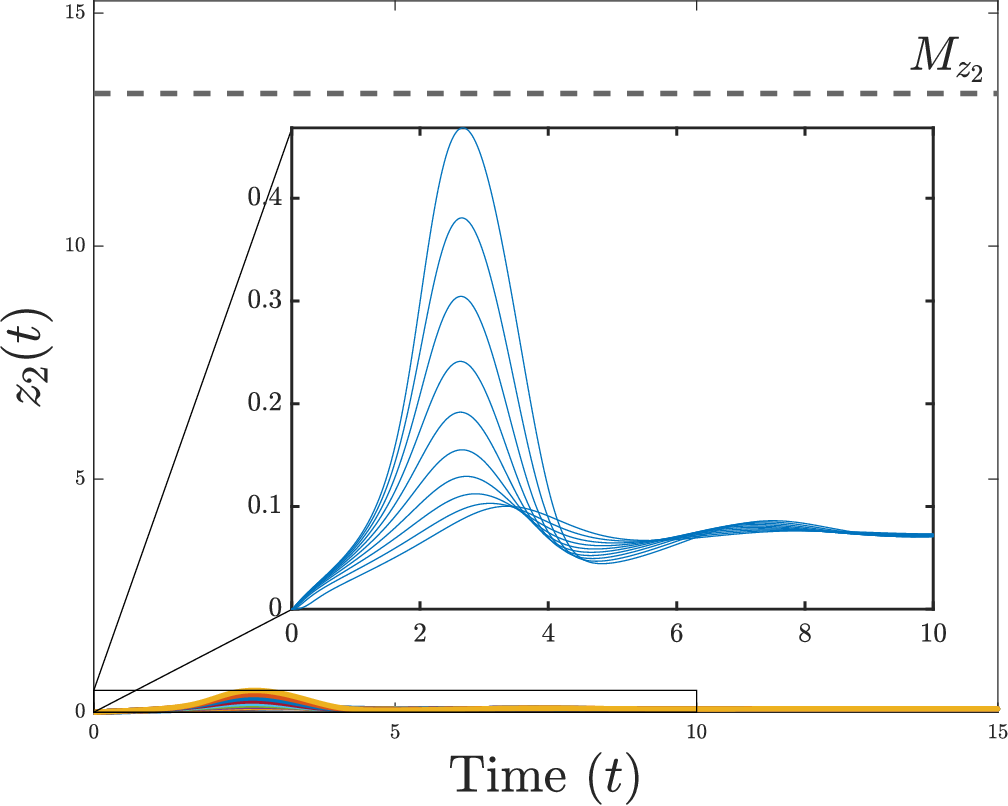}
    \label{fig:general_z2}
    }
    \caption{Numerical simulation for the general positive linear system with $n=12$.
    (a) Evolution of $z_1$ for ten initial conditions satisfying $0\le z_1(0)\le L_0$, together with the common theoretical bound $M_{z_1}$.
    (b) Evolution of $\|x(t)\|$ for the same initial conditions, together with the common bound $M_x$.
    (c) Evolution of $z_2$ for the same initial conditions, together with the common bound $M_{z_2}$.}
    \label{fig:general_numerical_results}
\end{figure}

\subsection{Positive linear system in $\mathbb{R}^2$}

As an example, consider the positive linear system in $\mathbb{R}^2$ given by
\eqref{eq:ca_dynamics}, with parameters $\alpha_i,\beta_i>0$ for
$i=1,\dots,4$. We choose the parameter values as in~\cite{AFC2016}, 
with $\alpha_2=\alpha_4=30$, $\beta_1=10,$
and $\alpha_1=\alpha_3=\beta_2=\beta_3=\beta_4=1$. The resulting system is
\begin{equation}\label{eq:ca_example}
    \begin{aligned}
    \dot z_1 &= 1-30 z_1 z_2,\\
    \dot z_2 &= x_2-30 z_1 z_2,
    \end{aligned}
    \qquad\text{and}\qquad
    \begin{aligned}
    \dot x_1 &= 10z_1-x_1,\\
    \dot x_2 &= x_1-x_2.
    \end{aligned}
\end{equation}
This system admits oscillatory solutions: see Figure~\ref{fig:numerical_results}(d) for a periodic orbit and Figure~\ref{fig:numerical_results}(b) for a solution approaching this periodic orbit when initialized at
$z_1(0)=z_2(0)=x_1(0)=x_2(0)=0$.
We compare the theoretical bounds with the simulated trajectories.

The computed threshold and waiting time are $L_0=1.5293$ and $T_0=1.3942,$
which yield
\begin{align*}
    \overline{M}_{z_1}&=L_0+\alpha_1T_0 \approx 2.9235 \\
    \overline{M}_{x_1}&=10\,\overline{M}_{z_1} \approx 29.2351 \\
    \overline{M}_{x_2}&=\overline{M}_{x_1}\approx 29.2351.
\end{align*}
Finally, since $K_* = \alpha_3 G/\alpha_4 = 1/3$, $c = 1$, and $d = 1$, we obtain
\[
    \overline{M}_{z_2}
    =K_*+c\overline{M}_{x_1}+d\overline{M}_{x_2}
    \approx 58.8036.
\]
Below are simulations from two different initial conditions,
\begin{align*}
    [z_1(0),\,z_2(0),\,x_1(0),\,x_2(0)]^\top&=[0,0,0,0]^\top \\
    [z_1(0),\,z_2(0),\,x_1(0),\,x_2(0)]^\top&=[0.043,\,0.756,\,0.332,\,0.407]^\top.
\end{align*}
We report only the bound $\overline{M}_{z_1}$ explicitly in
Figure~\ref{fig:numerical_results}(a) and (c), since the bounds
$\overline{M}_{x_1}$, $\overline{M}_{x_2}$, and $\overline{M}_{z_2}$
follow directly from $\overline{M}_{z_1}$ through the cascade comparison
arguments. In particular,
$\overline{M}_{x_1}$ and $\overline{M}_{x_2}$ scale proportionally with
$\overline{M}_{z_1}$, while $\overline{M}_{z_2}$ accumulates contributions
from the upstream states, making it the most conservative bound. This
explains why the theoretical bounds are considerably larger than the
observed trajectories.
\begin{figure}[h!]
    \centering
    \subfloat[]{
        \includegraphics[width=0.23\textwidth]{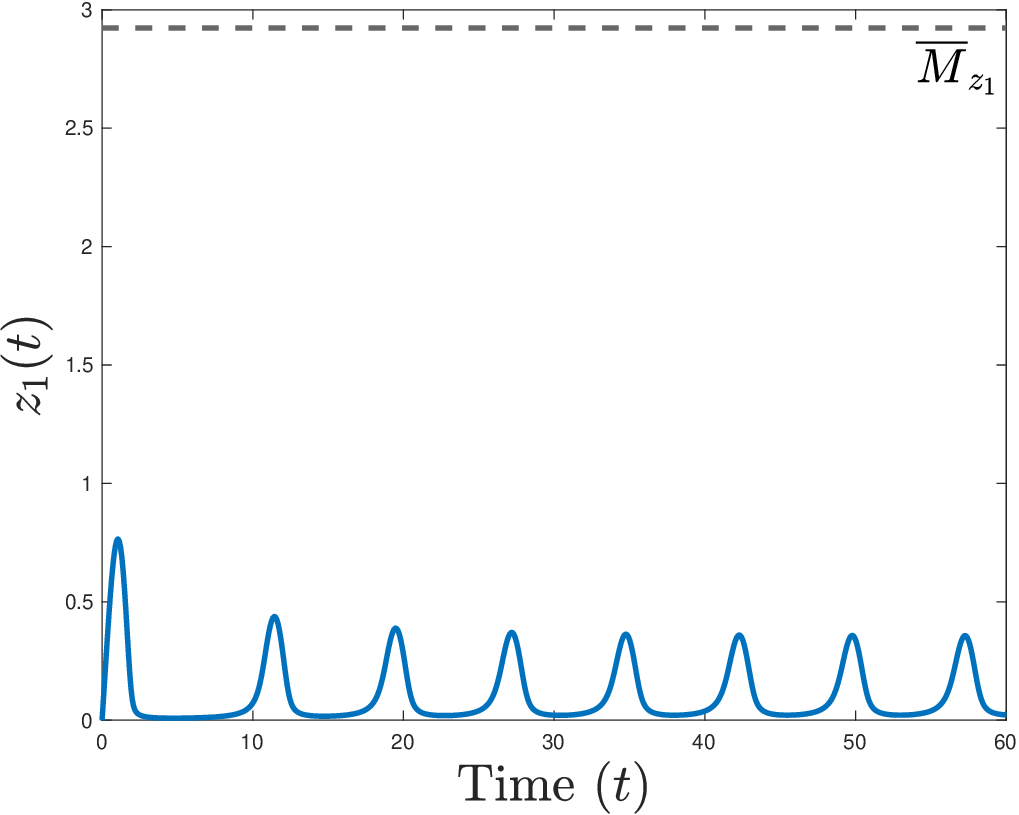}
        \label{fig:cbf_clf}
    }
    \subfloat[]{
        \includegraphics[width=0.23\textwidth]{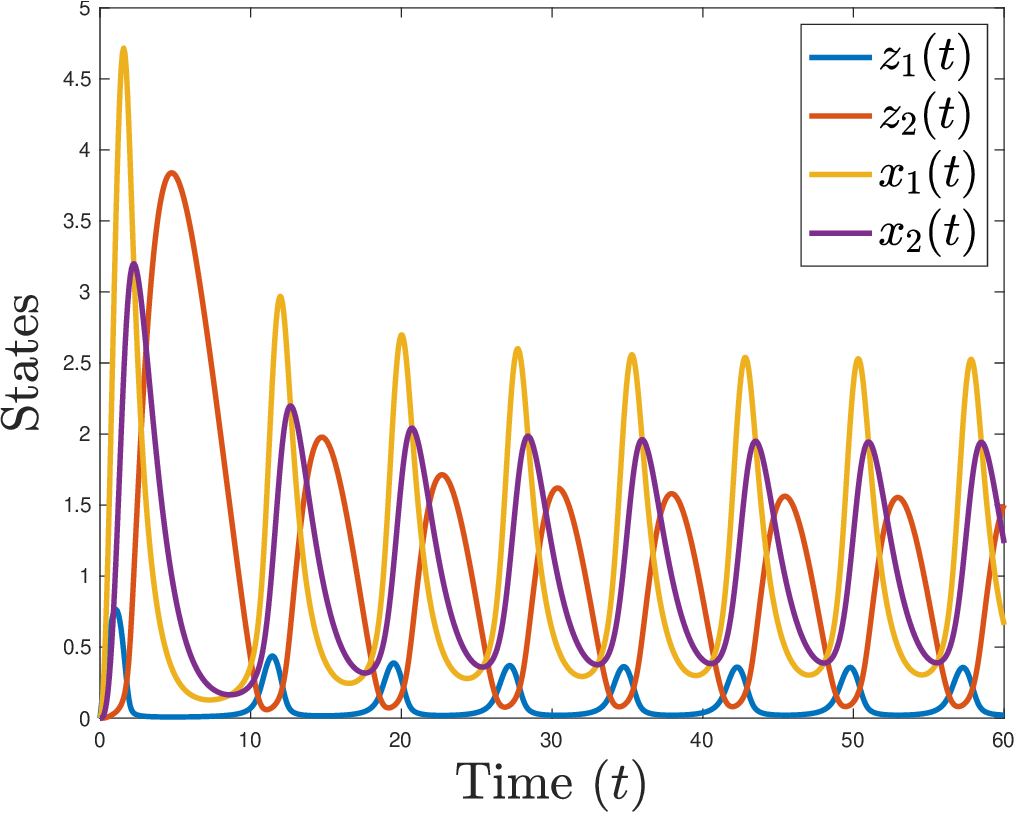}
        \label{fig:cbf_nom}
    }
    \subfloat[]{
        \includegraphics[width=0.23\textwidth]{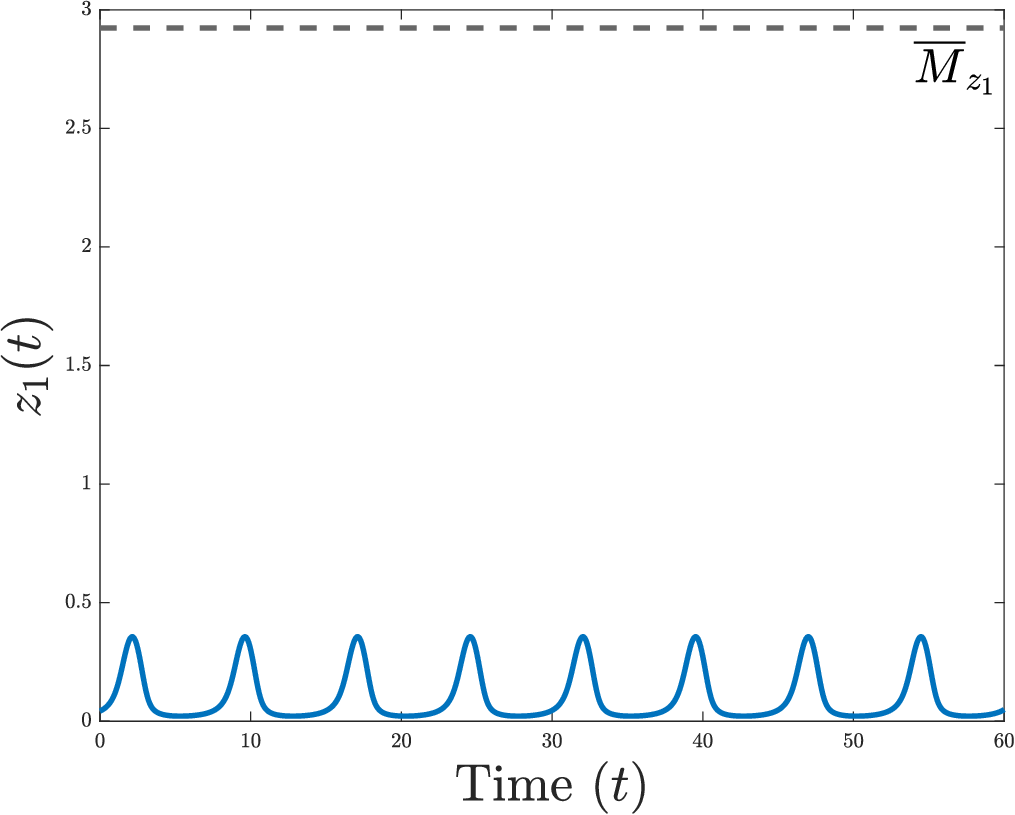}
        \label{fig:cbf_mpc}
    }
    \subfloat[]{
        \includegraphics[width=0.23\textwidth]{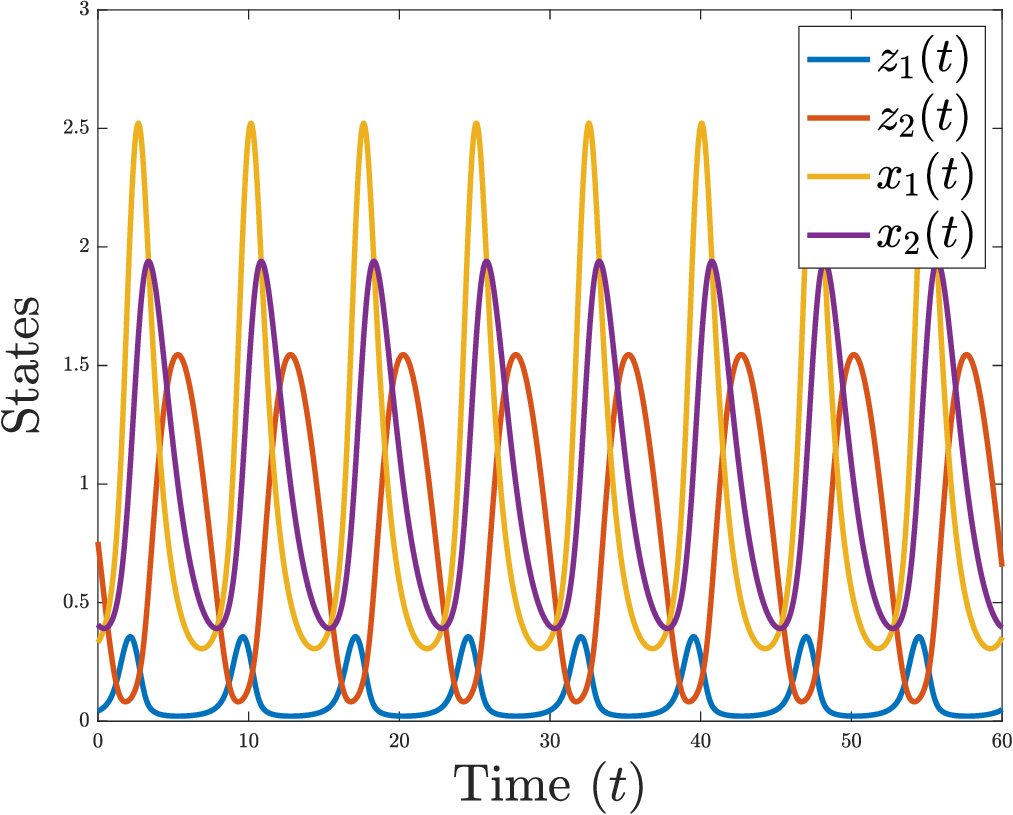}
        \label{fig:cbf_tubempc}
    }
    \caption{(a) Evolution of $z_1$ together with the theoretical bound $\overline{M}_{z_1}$ for the initial condition $[z_1(0),z_2(0),x_1(0),x_2(0)]^\top=[0,0,0,0]^\top$. 
    (b) Evolution of all states for the same initial condition.
    (c) Evolution of $z_1$ together with the theoretical bound $\overline{M}_{z_1}$ for the initial condition $[z_1(0),z_2(0),x_1(0),x_2(0)]^\top=[0.043,\,0.756,\,0.332,\,0.407]^\top$. 
    (d) Evolution of all states for the same initial condition.}
    \label{fig:numerical_results}
\end{figure}

\section{Conclusions}\label{sec:conclusion}

In this paper, we studied boundedness of a class of nonlinear feedback systems motivated by antithetic control. 
The key idea behind the boundedness proof was to view the dynamics of $z_1$ as a system subject to a negative feedback whose effective gain is determined by $z_2$. The crucial property is that whenever $z_1$ remains above a certain threshold $L_0$ for a sufficiently large time, the resulting growth of $z_2$ eventually generates a feedback action strong enough to overcome the positive drift of $z_1$. We may think of this mechanism as a form of \emph{persistently negative feedback}: a persistently large value of $z_1$ ultimately produces a corrective damping action that drives $z_1$ downward.

To establish this property, one must relate the behavior of $z_1$ to that of the intermediate subsystem described by the variables $x$, with input $u=z_1$ and output $y$. Intuitively, a persistently large input should produce a persistently large output, which in turn causes $z_2$ to increase and strengthen the negative feedback acting on $z_1$. A natural framework for formalizing this idea is provided by the theory of monotone control systems developed in \cite{monotoneTAC}. In particular, if for each constant input $u$ the subsystem admits a globally asymptotically stable equilibrium, yielding an input-state characteristic in the sense of that work, then persistent increases in the input are reflected in corresponding increases in the asymptotic state and output. Such a property can be established under a variety of technical assumptions. For simplicity, we restricted attention here to the case in which the intermediate subsystem is linear and positive, a special case of monotone systems. This restriction is motivated by our main application, where the $x$-subsystem is a two-dimensional positive linear system. 

Future work will extend the present analysis to broader classes of nonlinear systems and investigate the extent to which the boundedness mechanism identified here persists beyond the positive linear setting.
Other possible extensions include systems with time-varying parameters or disturbances, as well as studying convergence properties in addition to boundedness.

\appendix
\section{Technical results}\label{sec:technical}
\begin{proposition}\label{prop:T_fixedpoint}
    Let $\Delta>0$ and $\alpha_1,\alpha_4>0$. For each $\ell\in\mathbb{R}_+$, consider the fixed-point equation
    \begin{equation}\label{eq:T_fixedpoint}
        \tau(\ell) = \Delta+\frac{\ln(2)}{\alpha_4(\ell+\alpha_1\tau(\ell))}.
    \end{equation}
    Then, $\tau(\ell)$ satisfies:
    \begin{enumerate}
        \item[\emph{(i)}] the unique positive solution $\tau^* = \tau(\ell)$ exists;
        \item[\emph{(ii)}] $\tau(\ell)$ is continuous on $[0,\infty)$, differentiable on $(0,\infty)$, and strictly decreasing; moreover, the right derivative at $\ell=0$ exists and is negative;
        \item[\emph{(iii)}] $\tau(\ell)>\Delta$ for all $\ell\ge0$, and $\lim_{\ell\to\infty}\tau(\ell)=\Delta$.
        \item[\emph{(iv)}] the solution $\tau(\ell;\Delta)$ is strictly increasing in $\Delta$ for each fixed $\ell\ge0$.
    \end{enumerate}
\end{proposition}

\begin{proof}
    The uniqueness can be easily verified by noticing that the LHS of the equation above is the line with unit slope, while the RHS is a shifted hyperbola, implying they must cross exactly once for $\ell\ge0$. More precisely,
    setting $\psi := \ln(2)/\alpha_4$, 
    the fixed-point equation becomes $\tau(\ell) = \Delta + \psi/(\ell + \alpha_1 \tau(\ell))$. Equivalently $(\tau(\ell)-\Delta)(\ell+\alpha_1\tau(\ell))=\psi$ or
    $q(\tau(\ell)):=\alpha_1\tau(\ell)^2+(\ell-\alpha_1\Delta)\tau(\ell)-(\Delta\ell+\psi)=0$,
    which has exactly one positive root
    \begin{equation}\label{eq:T_root}
        \tau^*=\frac{\alpha_1\Delta-\ell+\sqrt{(\ell+\alpha_1\Delta)^2+4\alpha_1\psi}}{2\alpha_1}.
     \end{equation}
    Hence $\tau^* = \tau(\ell)$ is well defined, which proves (i). The fixed-point equation also gives
    \begin{equation*}
        \tau(\ell)-\Delta = \frac{\ln(2)}{\alpha_4(\ell+\alpha_1\tau(\ell))} > 0
        \qquad\text{for every } \ell\ge0.
    \end{equation*}
    The explicit formula \eqref{eq:T_root} shows that $\tau$ is continuous on $[0,\infty)$ and differentiable on $(0,\infty)$. Moreover, the derivative extends continuously to $\ell=0$, where it equals the right derivative.
    Differentiating $q(\tau(\ell))$ with respect to $\ell$ yields
    \begin{equation*}
        2\alpha_1\tau(\ell)\tau'(\ell)+\tau(\ell)+(\ell-\alpha_1\Delta)\tau'(\ell)-\Delta=0,
    \end{equation*}
    from which it follows that
    \begin{equation*}
        \tau'(\ell)=\frac{\Delta-\tau(\ell)}{2\alpha_1\tau(\ell)+\ell-\alpha_1\Delta} 
        = -\frac{\ln(2)}{\alpha_4\left(\ell+\alpha_1\tau(\ell)\right)^2+\alpha_1\ln(2)}<0.
    \end{equation*}
    The formula is valid for all $\ell>0$ and extends continuously to $\ell=0$, where it gives the right derivative.
    Since $\tau(\ell)>\Delta$, it follows that $\tau'(\ell)<0$. Therefore $\tau(\ell)$ is strictly decreasing, proving (ii). Furthermore, using the expansion
    \footnote{Let $a:=\ell+\alpha_1\Delta$. Then $\sqrt{a^2+4\alpha_1\psi}=a\sqrt{1+4\alpha_1\psi/a^2}.$ Since $a\to\infty$ as $\ell\to\infty$, the quantity $x:=4\alpha_1\psi/a^2$ satisfies $x\to0$. Using the Taylor expansion $\sqrt{1+x}=1+x/2+o(x)$ as $x\to0$, we obtain
    $\sqrt{a^2+4\alpha_1\psi}=a+2\alpha_1\psi/a+o\!\left(1/a\right).$
    Substituting back $a=\ell+\alpha_1\Delta$ yields the stated expansion.}
    \begin{equation*}
        \sqrt{(\ell+\alpha_1\Delta)^2+4\alpha_1\psi}
        =\ell+\alpha_1\Delta+\frac{2\alpha_1\psi}{\ell+\alpha_1\Delta}
        +o\!\left(\frac1\ell\right), \qquad \ell\to\infty,
    \end{equation*}
    and substituting this into the expression for $\tau(\ell)$ gives, as $\ell\to\infty$, $\lim_{\ell\to\infty}\tau(\ell)=\Delta$.
    Together with the inequality $\tau(\ell)>\Delta$ proved above, this proves point (iii).

    Finally, to prove (iv), view the solution as $\tau=\tau(\ell;\Delta)$ and 
    differentiate the fixed-point equation with respect to $\Delta$. This gives
    \begin{equation*}
        \frac{\partial \tau}{\partial \Delta}
        =
        1-
        \frac{\alpha_1\ln(2)}{\alpha_4(\ell+\alpha_1\tau)^2}
        \frac{\partial \tau}{\partial \Delta}
        =
        \frac{1}{
        1+\dfrac{\alpha_1\ln(2)}
        {\alpha_4(\ell+\alpha_1\tau)^2}}
        >0.
    \end{equation*}
    Therefore, for each fixed $\ell\ge0$, $\tau(\ell;\Delta)$ is strictly increasing in $\Delta$, proving (iv).
\end{proof}

\begin{proposition}\label{prop:ell4_monotone}
Let $\tau(\ell):\mathbb{R}_+\to\mathbb{R}_+$ be defined as in 
\eqref{eq:T_fixedpoint}, $\alpha_3>0$ and $G\coloneqq-c^\top A^{-1}b>0$. For each $\ell\in\mathbb{R}_+$, define
\begin{equation*}
    \rho_z(\ell,\tau(\ell))
    := \frac{\alpha_3 G\ell}{4\alpha_4(\ell+\alpha_1\tau(\ell))}.
\end{equation*}
Then $\rho_z:\mathbb{R}_+\to\mathbb{R}_+$ is strictly increasing with 
$\lim_{\ell\to\infty}\rho_z(\ell,\tau(\ell))=\alpha_3G/(4\alpha_4)$.
\end{proposition}

\begin{proof}
Denote
$\alpha \coloneqq \alpha_3 G/(4\alpha_4) > 0$. From the definition of $\rho_z(\ell,\tau(\ell))$, it suffices to study the function
\begin{equation*}
    \phi(\ell):=\frac{\ell}{\ell+\alpha_1\tau(\ell)},
\end{equation*}
since $\rho_z(\ell,\tau(\ell))=\alpha\phi(\ell)$. 
Differentiating, we obtain
\begin{equation*}
    \phi'(\ell)
    =\frac{(\ell+\alpha_1\tau(\ell)) - \ell(1+\alpha_1\tau'(\ell))}
          {(\ell+\alpha_1\tau(\ell))^2}
    =\frac{\alpha_1\bigl(\tau(\ell)-\ell\tau'(\ell)\bigr)}
          {(\ell+\alpha_1\tau(\ell))^2}.
\end{equation*}
By Proposition~\ref{prop:T_fixedpoint}, $\tau(\ell)>0$ and $\tau'(\ell)<0$ 
for all $\ell\ge0$, hence $\tau(\ell)-\ell\tau'(\ell)>0$, and therefore 
$\phi'(\ell)>0$ for all $\ell\ge0$. Thus $\phi(\ell)$, and hence 
$\rho_z(\ell,\tau(\ell))$, is strictly increasing on $[0,\infty)$.
Moreover, since Proposition~\ref{prop:T_fixedpoint} gives 
$\lim_{\ell\to\infty}\tau(\ell)=\Delta<\infty$, it follows that
$\lim_{\ell\to\infty}\phi(\ell)=1$, and hence $\lim_{\ell\to\infty}\rho_z(\ell,\tau(\ell))=\alpha_3G/(4\alpha_4)$.
\end{proof}

\begin{proposition}\label{prop:Lell4_threshold}
Let $\rho_z(\ell,\tau(\ell))$ be defined as in Proposition~\ref{prop:ell4_monotone}. Then the map 
$\ell\mapsto \ell\rho_z(\ell,\tau(\ell))$ is strictly increasing 
on $[0,\infty)$ and satisfies
\begin{equation*}
    \lim_{\ell\to\infty}\ell\rho_z(\ell,\tau(\ell))=\infty.
\end{equation*}
Consequently, for every $a>0$, there exists a unique $\ell^*>0$ such that 
$\ell^*\rho_z(\ell^*,\tau(\ell^*))=a$.
\end{proposition}

\begin{proof}
By Proposition~\ref{prop:ell4_monotone}, $\rho_z(\ell,\tau(\ell))\ge0$ 
for all $\ell\ge0$ and $\rho_z$ is strictly increasing, so that 
$\ell\mapsto\ell\rho_z(\ell,\tau(\ell))$ is strictly increasing 
on $[0,\infty)$. Moreover, Proposition~\ref{prop:ell4_monotone} gives 
$\lim_{\ell\to\infty}\rho_z(\ell,\tau(\ell))=\alpha_3G/(4\alpha_4)>0$, 
and it follows that $\lim_{\ell\to\infty}\ell\rho_z(\ell,\tau(\ell))
=\infty$. Since $\ell\mapsto\ell\rho_z(\ell,\tau(\ell))$ is 
continuous, strictly increasing, and satisfies
\begin{equation*}
    \lim_{\ell\to 0^+}\ell\rho_z(\ell,\tau(\ell))=0,
    \qquad
    \lim_{\ell\to\infty}\ell\rho_z(\ell,\tau(\ell))=\infty,
\end{equation*}
the intermediate value theorem implies that for every $a>0$ there exists 
a unique $\ell^*>0$ such that $\ell^*\rho_z(\ell^*,\tau(\ell^*))=a$.
Taking $a=\alpha_1/\alpha_2$ gives the desired threshold $L_0$.
\end{proof}

\begin{lemma}\label{lem:p_above_theta}
    Let $\theta:=\alpha_1/\alpha_2$. Consider the dynamics of $z_1(t)$ and $z_2(t)$ in \eqref{eq:gen_dynamics}. Suppose that $z_1(t)\ge L>L_0$ for all $t\in[0,T]$ with $T\ge T_0$. Assume, in addition, that $p(T_0)>\theta$. Then $p(t)>\theta$ for all $t\in[T_0,T]$.
\end{lemma}

\begin{proof}[Proof of Lemma~\ref{lem:p_above_theta}]
    Recall that $\theta:=\alpha_1/\alpha_2$. We prove that $p(t)>\theta$ for all $t\in[T_0,T]$
    by contradiction. Suppose instead that there exists $t\in(T_0,T]$ such that $p(t)\le\theta.$
    Since $p(T_0)>\theta$, the
    intermediate value theorem implies that there exists $\bar{t}\in(T_0,t]$ such that $p(\bar{t})=\theta$.
    
    Let $t_0$ denote the first time at which this occurs, i.e., $t_0:=\min\{s\in(T_0,T]:p(s)=\theta\}.$
    Then $p(t)>\theta$ for all $t\in[T_0,t_0)$ and $p(t_0)=\theta$. Hence, the first--crossing argument in Lemma~\ref{lem:first_hitting} implies that
    \begin{equation*}
        \dot p(t_0)\le 0.
    \end{equation*}
    However, notice that since $p(t_0)=\theta$ by definition of $t_0$, $\dot z_1(t_0)=\alpha_1-\alpha_2\theta=0$, which implies that 
    \begin{align*}
        \dot p(t_0) = z_1(t_0)\dot z_2(t_0) = \bigl[\alpha_3 y(t_0)-\alpha_4\theta\bigr]z_1(t_0).
    \end{align*}
    Since $z_1(t_0) \ge L > 0$, the sign of $\dot p(t_0)$ is determined by the term in parentheses. Since $t_0>T_0>\Delta$, we have $y(t_0)\ge\rho_y(L)$ from \eqref{eq:y_low}, and since $\rho_y(L)=GL/2$ is increasing in $L$ and $L>L_0$,
    \begin{equation*}
        \alpha_3 y(t_0)-\alpha_4\theta 
        \ge \alpha_3\rho_y(L)-\alpha_4\theta
        > \alpha_3\rho_y(L_0)-\alpha_4\theta.
    \end{equation*}
    It remains to show that $L_0$ satisfies $\alpha_3 GL_0/2>\alpha_4\theta$, 
    or equivalently, 
    \begin{equation*}
        L_0>L^*:=2\alpha_4\theta/(\alpha_3 G).
    \end{equation*}
    Notice that $L_0\rho_z(L_0,T_0) = \theta$
    by definition of $L_0$ in \eqref{eq:L0_def}. Since $\ell\rho_z(\ell,\tau(\ell))$ is strictly 
    increasing in $\ell$ (Proposition~\ref{prop:Lell4_threshold}), it suffices to show that
    $L^*\rho_z(L^*,\tau(L^*))<\theta$, which would imply $L_0>L^*$. Indeed,
    \begin{align*}
        L^*\rho_z(L^*,\tau(L^*))
        &= \frac{\alpha_3 G(L^*)^2}{4\alpha_4\bigl(L^*+\alpha_1\tau(L^*)\bigr)}
        < \frac{\alpha_3 G(L^*)^2}{4\alpha_4 L^*}
        = \frac{\alpha_3 G L^*}{4\alpha_4}
        = \frac{\theta}{2}
        < \theta,
    \end{align*}
    where we used $\tau(L^*)>0$ in the first inequality and substituted 
    $L^*=2\alpha_4\theta/(\alpha_3 G)$ in the last equality.
    Hence $L_0>L^*$, so $\alpha_3GL_0/2>\alpha_4\theta$, and 
    therefore
    \begin{equation*}
        \dot p(t_0)=z_1(t_0)\left[\alpha_3 y(t_0)-\alpha_4\theta\right]>0,
    \end{equation*}
    contradicting $\dot p(t_0)\le 0$.
\end{proof}

\begin{lemma}[First hitting-time]\label{lem:first_hitting}
Let $f:[0,\infty)\to\mathbb{R}$ be continuous and differentiable, and let
$Q\in\mathbb{R}$. Suppose that the set $\{t\ge 0:f(t)=Q\}$ is nonempty and define $t^\ast:=\min\{t\ge 0:f(t)=Q\}.$
\begin{enumerate}
    \item If $f(0)<Q$, then $f(t)\le Q$ for all $t\in[0,t^\ast],$ and $\dot f(t^\ast)\ge 0.$
    Moreover, if $f$ crosses $Q$ transversally from below at $t^\ast$, i.e., $\dot f(t^\ast)\neq 0$, then $\dot f(t^\ast)>0$.
    
    \item If $f(0)>Q$, then $f(t)\ge Q$ for all $t\in[0,t^\ast],$ and $\dot f(t^\ast)\le 0.$
    Moreover, if $f$ crosses $Q$ transversally from above at $t^\ast$, i.e., $\dot f(t^\ast)\neq 0$, then $\dot f(t^\ast)<0.$
\end{enumerate}
\end{lemma}

\begin{proof}
We prove the case $f(0)<Q$; the case $f(0)>Q$ is analogous with the inequalities reversed.
Since $t^\ast$ is the first time at which $f$ reaches the level $Q$ and
$f(0)<Q$, we have $f(t)<Q$ for all $t\in[0,t^\ast)$ and $f(t^\ast)=Q.$
Hence $f(t)\le Q$ for all $t\in[0,t^\ast]$.

For $h<0$ sufficiently small, we have $t^\ast+h<t^\ast$, and so $ f(t^\ast+h)<Q=f(t^\ast).$
Therefore, $\frac{1}{h}[f(t^\ast+h)-f(t^\ast)] \ge 0$. Passing to the limit as $h\to 0^-$ gives $\dot f(t^\ast)\ge 0.$
If, in addition, $f$ crosses $Q$ transversally from below at $t^\ast$, then
$\dot f(t^\ast)\neq 0$. Since $\dot f(t^\ast)\ge 0$, it follows that $\dot f(t^\ast)>0.$
\end{proof}

\section{Additional technical proofs}

\begin{proof}[Proof of Lemma~\ref{lem:GenLin-OctInv}]\label{pr:lemma}

    To prove invariance of the positive orthant $\mathcal{C}$, we invoke the Bony-Brezis Theorem, which states that a closed set is forward invariant under a vector field $f(\cdot)$ (i.e. solutions of that vector field initialized in the set remain in the set) if and only if $\langle f(w),v\rangle\leq0$ for every exterior normal vector $v$ at any boundary point $w$ of $\mathcal{C}$. To do that, first formally define for \eqref{eq:gen_dynamics}
    \begin{equation*}
        f(w) = f([z_1;z_2;x]):=\begin{bmatrix}
            \alpha_1-\alpha_2z_1z_2 \\ \alpha_3c^\top x-\alpha_4z_1z_2 \\ Ax+bz_1
        \end{bmatrix},
    \end{equation*}
    and let $\partial\mathcal{C}$ denote the boundary of $\mathcal{C}$. 
    Note that
    \begin{equation*}
        \partial\mathcal{C} = \{(z_1,z_2,x)\in\mathcal{C} : 
        \text{at least one coordinate of } (z_1,z_2,x) \text{ is zero}\},
    \end{equation*}
    since if all coordinates are strictly positive, then a small neighborhood 
    still lies entirely in $\mathcal{C}$, so the point is in the interior.

    Next, at any point $w\in\partial\mathcal{C}$ let $I(w):=\{i\in\{1,\dots,n+2\}~|~w_i=0\}$ and denote $I(w)=I$ if the point $w$ is clear by context. We will prove that the exterior normal cone of $\mathcal{C}$ at $w$ is given by $\{v\in\mathbb{R}^{n+2}~|~v=-\sum_{i\in I}\lambda_ie_i,~\lambda_i\ge0\}$. For any closed convex set (such as the positive orthant) the exterior normal cone can be defined as $\mathcal{N}_{\mathcal{C}}(w):=\{v\in\mathbb{R}^{n+2}~|~\langle v,\xi-w\rangle\leq 0~\text{for all }\xi\in\mathcal{C}\}$. Pick any $v\in\mathcal{N}_{\mathcal{C}}(w)$ and consider two cases: first, if $i\not\in I$, then $w_i>0$ and thus for sufficiently small $\varepsilon>0$, $w\pm\varepsilon e_i\in \mathcal{C}$. Using the definition of $\mathcal{N}_{\mathcal{C}}(w)$ with $\xi=w+\varepsilon e_i$ gives that $\langle v,\varepsilon e_i\rangle \leq 0$ and thus $v_i\leq 0$. Repeating the same logic for $\xi=w-\varepsilon e_i$ gives $v_i\geq 0$ and thus $v_i=0$; second, if $i\in I$ then for any sufficiently small $\varepsilon>0$, $w+\varepsilon e_i\in\mathcal{C}$, and again using the definition of the normal cone gives $v_i\leq 0$. This means $v_i=0$ if $i\not\in I$ and $v_i\leq 0$ if $i\in I$, thus $v=-\sum_{i\in I}\lambda_i e_i$, $\lambda_i:=-v_i\geq 0$, and thus $\mathcal{N}_{\mathcal{C}}(w)\subseteq\{v\in\mathbb{R}^{n+2}~|~v=-\sum_{i\in I}\lambda_ie_i,~\lambda_i\ge0\}$.

    To prove the converse pick any $v=-\sum_{i\in I}\lambda_ie_i$, with $\lambda_i\ge0$. Then for any $\xi\in \mathcal{C}$, $\langle v,\xi-w\rangle = -\sum_{i\in I}\lambda_i(\xi_i-w_i)=-\sum_{i\in I}\lambda_i\xi_i\leq 0$, proving that $\mathcal{N}_{\mathcal{C}}(w)\supseteq\{v\in\mathbb{R}^{n+2}~|~v=-\sum_{i\in I}\lambda_ie_i,~\lambda_i\ge0\}$ and proving equality.

    Finally we prove invariance of the positive orthant under \eqref{eq:gen_dynamics}. First, notice that any point $w=(z_1,z_2,x)$ satisfying $z_1=0$ is such that $\{f(w)\}_1 = \dot z_1|_{w} = \alpha_1-\alpha_2z_1z_2 = \alpha_1>0$, and thus $\langle f(w),-e_1\rangle = -\alpha_1<0 $. Next, notice that since $c\geq 0$ by assumption, at any point $w\in\partial \mathcal{C}$, $x\geq 0$ in particular, implying that $c^\top x\geq 0$. Thus, if $z_2=0$ at $w$, then $\langle f(w),-e_2\rangle = -\alpha_3y+\alpha_4z_1z_2 = -\alpha_3y\leq 0$. Finally, for any $i\in I$ assume $w$ is such that $x_{i-2}=0$ and compute $\langle f(w),-e_i \rangle = -(\sum_{j=1}^{n}A_{(i-2),j}x_j+b_{i-2}z_1)$. Since $b\geq 0$ and $z_1\geq0$ at $w$, $b_{i-2}z_1\geq0$. Furthermore, since $A$ is Metzler and $x_{i-2}=0$ then $\sum_{j=1}^{n}A_{(i-2),j}x_j\ge0$, implying that $\langle f(w),-e_i \rangle\leq 0$. We, thus, have proven that for all $w\in\partial \mathcal{C}$ and all $i\in I=\{i\in\{1,\dots, n+2\}~|~w_i=0\}$, $\langle f(w),-e_i\rangle \le0$. By conic combination, we conclude that for all $v\in\mathcal{N}_{\mathcal{C}}(w)$, $\langle f(w),v\rangle=-\sum_{i\in I}\lambda_i\langle f(w),e_i\rangle\leq 0$, concluding the proof.
\end{proof}

\begin{proof}[Proof of Corollary~\ref{cor:z1_overshoot_4d}]
    The proof is identical to Lemma~\ref{lem:z1_overshoot}.
\end{proof}

\begin{proof}[Proof of Corollary~\ref{cor:z1_bounded_4d}]
    The proof is identical to that of Theorem~\ref{thm:z1_bounded}, using 
    Corollary~\ref{cor:z1_overshoot_4d} in place of Lemma~\ref{lem:z1_overshoot}.
\end{proof}

\begin{proof}[Proof of Corollary~\ref{cor:x_bounded_4d}]
    Since $z_1(t)\le \overline{M}_{z_1}$ for all $t\ge 0$, this
    gives $\dot x_1\le\beta_1\overline{M}_{z_1}-\beta_2 x_1$. Let $w_1$ solve 
    $\dot w_1=\beta_1\overline{M}_{z_1}-\beta_2 w_1$ with $w_1(0)=x_1(0)$. Its explicit 
    solution is
    \begin{equation*}
        w_1(t)=\frac{\beta_1\overline{M}_{z_1}}{\beta_2}
        +\Bigl[x_1(0)-\frac{\beta_1\overline{M}_{z_1}}{\beta_2}\Bigr]e^{-\beta_2 t}.
    \end{equation*}
    By the comparison principle, $x_1(t)\le w_1(t)$ for all $t\ge 0$. Since 
    $w_1(t)$ remains between $w_1(0)$ and its equilibrium $\beta_1\overline{M}_{z_1}/\beta_2$,
    \begin{equation}\label{eq:x1_bound}
        x_1(t)\le \max\left\{x_1(0),\,\frac{\beta_1}{\beta_2}\overline{M}_{z_1}\right\}
        \eqqcolon \overline{M}_{x_1} \qquad\text{for all } t\ge 0.
    \end{equation}
    Next, using $x_1(t)\le \overline{M}_{x_1}$, the equation for $x_2$ gives 
    $\dot x_2\le\beta_3\overline{M}_{x_1}-\beta_4 x_2$. Let $w_2$ solve 
    $\dot w_2=\beta_3\overline{M}_{x_1}-\beta_4 w_2$ with $w_2(0)=x_2(0)$. Then
    \begin{equation*}
        w_2(t)=\frac{\beta_3\overline{M}_{x_1}}{\beta_4}
        +\Bigl[x_2(0)-\frac{\beta_3\overline{M}_{x_1}}{\beta_4}\Bigr]e^{-\beta_4 t}.
    \end{equation*}
    By the comparison principle, $x_2(t)\le w_2(t)$ for all $t\ge 0$. Hence
    \begin{equation}\label{eq:x2_bound}
        x_2(t)\le \max\left\{x_2(0),\,\frac{\beta_3}{\beta_4}\overline{M}_{x_1}\right\}
        \eqqcolon \overline{M}_{x_2} \qquad\text{for all } t\ge 0.
    \end{equation}
    Thus $x_1$ and $x_2$ are bounded on $[0,\infty)$.
\end{proof}

\begin{proof}[Proof of Corollary~\ref{cor:z2_bounded_4d}]
    The proof is identical to that of Lemma~\ref{lem:z2_bounded}, with $x$ 
    replaced by $[x_1,\,x_2]^\top$, $\nu^\top$ replaced by $[c,\,d]$, $K^*$ 
    replaced by $K_*:=\alpha_3 G/\alpha_4$ with $G=\beta_1\beta_3/(\beta_2\beta_4)$, 
    and $\gamma$ replaced by $\bar\gamma:=K_*+c\overline{M}_{x_1}+d\overline{M}_{x_2}$. 
    The key computation is that by the choice of $c$ and $d$, the $x_1$- and 
    $x_2$-terms cancel in $\dot W$, giving $\dot W=\alpha_4 z_1(K_*-z_2)$. From this, 
    one shows that $W$ cannot exceed $\bar\gamma$: whenever $W(t)>\bar\gamma$, 
    we have $z_2(t)>K_*$, which forces $\dot W(t)\le 0$, preventing further 
    growth. It then follows that $W(t)\le\max\{W(0),\bar\gamma\}$ for all 
    $t\ge 0$.
    Since $c,d\ge0$ and $x_1(t),x_2(t)\ge0$, we have $z_2(t)=W(t)-cx_1(t)-dx_2(t)\le W(t)$. Therefore
    \begin{equation*}
        z_2(t) \le W(t) \le \max\{W(0),\bar\gamma\}\eqqcolon \overline{M}_{z_2} 
        \qquad\text{for all }t\ge 0
    \end{equation*}
    completing the proof.
\end{proof}

\newpage

\bibliographystyle{ieeetr}
\bibliography{ref}

@article{Huang2018,
  author  = {Huang, H-H and Qian, Y and Del Vecchio, D},
  title   = {A quasi-integral controller for adaptation of genetic modules to variable ribosome demand},
  journal = {Nature Communications},
  volume  = {9},
  number  = {1},
  pages   = {5415},
  year    = {2018},
  doi     = {10.1038/s41467-018-07899-z},
  url     = {https://doi.org/10.1038/s41467-018-07899-z}
}

@book {mct,
    AUTHOR = {Sontag, E.D.},
 TITLE = {Mathematical {C}ontrol {T}heory. {D}eterministic {F}inite-{D}imensional {S}ystems},
    SERIES = {Texts in Applied Mathematics},
    VOLUME = {6},
   EDITION = {Second},
 PUBLISHER = {Springer-Verlag},
   ADDRESS = {New York},
      YEAR = {1998},
     PAGES = {xvi+531},
      ISBN = {0-387-98489-5},
      }

@INPROCEEDINGS{Margaliot-CDC,
  author={Margaliot, Michael and Wu, Chengshuai and Sontag, Eduardo D.},
  booktitle={2025 IEEE 64th Conference on Decision and Control (CDC)}, 
  title={Compact attractors of an antithetic integral feedback system have a simple structure}, 
  year={2025},
  volume={},
  number={},
  pages={2880-2885},
  doi={10.1109/CDC57313.2025.11312315}}

@article{rev_internal_model_2022,
   author = "Bin, Michelangelo and Huang, Jie and Isidori, Alberto and Marconi, Lorenzo and Mischiati, Matteo and Sontag, Eduardo",
   title = "Internal Models in Control, Bioengineering, and Neuroscience", 
   journal= "Annual Review of Control, Robotics, and Autonomous Systems",
   year = "2022",
   volume = "5",
   pages = "55-79",
  }

@article{Khammash2019,
author={Stephanie K. Aoki and  Gabriele Lillacci and  Ankit Gupta and  Armin Baumschlager and  David Schweingruber and  Mustafa Khammash},
title={A universal biomolecular integral feedback controller for robust perfect adaptation},
journal={Nature}, volume={570}, pages={533-537},
year={2019},
}

@article{agarwal2019naturecom,
  author = {D. K. Agrawal and R. Marshall and V. Noireaux and E. D. Sontag},
  title = {In vitro implementation of robust gene regulation in a synthetic biomolecular integral controller},
  journal={Nature Communications},
  year={2019},volume={10},pages={5760}, 
}

@misc{briat2025structuralstabilitypropertiesantithetic,
      title={Structural Stability Properties of Antithetic Integral (Rein) Control with Output Inhibition}, 
      author={Corentin Briat and Mustafa Khammash},
      year={2025},
      eprint={2201.13375},
      archivePrefix={arXiv},
      primaryClass={math.OC},
      url={https://arxiv.org/abs/2201.13375}, 
}

@article{AFC2016,
author={Briat, C. and  Gupta, A.  and  Khammash, M.}, title={Antithetic integral feedback ensures robust
perfect adaptation in noisy biomolecular networks},
journal={Cell Syst.},
volume={2}, pages={15-26}, year={2016},
}

@article{Eyal_k_posi,
	title={A Generalization of Linear Positive Systems with Applications to Nonlinear Systems: Invariant Sets and the {Poincar\'{e}-Bendixson} Property},
	author={E. Weiss and M. Margaliot},
	year={2021}, volume = {123},
    pages = {109358},
	journal={Automatica},  
}

@book{farina2000positive,
  author    = {Farina, Lorenzo and Rinaldi, Sergio},
  title     = {Positive Linear Systems: Theory and Applications},
  publisher = {Wiley-Interscience},
  year      = {2000},
  series    = {Pure and Applied Mathematics},
  address   = {New York},
  isbn      = {978-0-471-38456-5}
}

@article {monotoneTAC,
    AUTHOR = {Angeli, D. and Sontag, E.D.},
 TITLE = {Monotone control systems},
   JOURNAL = {IEEE Trans. Automat. Control},
    VOLUME = {48},
      YEAR = {2003},
    NUMBER = {10},
     PAGES = {1684--1698},
      ISSN = {0018-9286},
}

@article{2019biorxiv_margaliot_sontag,
	author = {Margaliot, M. and Sontag, E.D.},
	title = {Compact attractors of an antithetic integral feedback system have a simple structure (preprint version)},
	note = {DOI: 10.1101/868000},
	elocation-id = {868000},
	year = {2019},
	doi = {10.1101/868000},
	publisher = {Cold Spring Harbor Laboratory},
	URL = {https://www.biorxiv.org/content/early/2019/12/08/868000},
	eprint = {https://www.biorxiv.org/content/early/2019/12/08/868000.full.pdf},
	journal = {bioRxiv}
}

@article{katz2025instability,
author  = {Katz, Rami and Giordano, Giulia and Margaliot, Michael},
title   = {Instability of equilibrium and convergence to periodic orbits in strongly 2-cooperative systems},
journal = {Journal of Differential Equations},
volume  = {444},
pages   = {113651},
year    = {2025},
doi     = {10.1016/j.jde.2025.113651},
url     = {https://doi.org/10.1016/j.jde.2025.113651}
}

@article{olsman2019antithetic,
title = {Antithetic integral feedback for the robust control of monostable and oscillatory biomolecular circuits},
journal = {IFAC-PapersOnLine},
volume = {53},
number = {2},
pages = {16826-16833},
year = {2020},
note = {21st IFAC World Congress},
author = {Noah Olsman and Fulvio Forni},
}

\end{document}